\documentclass[leqno,11pt]{amsart}
\usepackage{amssymb, amsmath}
\def\norm#1#2{\|#1\|_{#2}}

\def\refer#1{~\ref{#1}}
\def\refeq#1{~(\ref{#1})}
\def\ccite#1{~\cite{#1}}

\def\suite#1#2#3{(#1_{#2})_{#2\in {#3}}}
\def\longformule#1#2{
\displaylines{ \qquad{#1} \hfill\cr \hfill {#2} \qquad\cr } }
\def\inte#1{
\displaystyle\mathop{#1\kern0pt}^\circ }

\def\sumetage#1#2{\sum_{\substack{{#1}\\{#2}}}}

\def\supetage#1#2{\sup_{\substack{{#1}\\{#2}}}}

\let\al=\alpha
\let\b=\beta
\let\g=\gamma
\let\d=\delta
\let\e=\varepsilon
\let\ep=\varepsilon

\let\lam=\lambda

\let\s=\sigma

\let\f=\phi

\let\D=\Delta

\let\wt=\widetilde
\let\wh=\widehat

\let\ov\overline

\def\cB{{\mathcal B}}
\def\cC{{\mathcal C}}
\def\cD{{\mathcal D}}

\def\cF{{\mathcal F}}
\def\cG{{\mathcal G}}

\def\cI{{\mathcal I}}
\def\cJ{{\mathcal J}}
\def\cK{{\mathcal K}}

\def\cM{{\mathcal M}}

\def\cO{{\mathcal O}}
\def\cP{{\mathcal P}}

\def\cS{{\mathcal S}}
\def\cT{{\mathcal T}}
\def\cU{{\mathcal U}}

\def\cW{{\mathcal W}}
\def\cX{{\mathcal X}}

\def\virgp{\raise 2pt\hbox{,}}
\def\cdotpv{\raise 2pt\hbox{;}}
\def\eqdefa{\buildrel\hbox{\footnotesize def}\over =}

\def\sgn{\mathop{\rm sgn}\nolimits}

\def\C{\mathop{\mathbb C\kern 0pt}\nolimits}
\def\DD{\mathop{\mathbb D\kern 0pt}\nolimits}
\def\EE{\mathop{{\mathbb E \kern 0pt}}\nolimits}
\def\K{\mathop{\mathbb K\kern 0pt}\nolimits}
\def\N{\mathop{\mathbb N\kern 0pt}\nolimits}
\def\Q{\mathop{\mathbb Q\kern 0pt}\nolimits}
\def\R{{\mathop{\mathbb R\kern 0pt}\nolimits}}
\def\SS{\mathop{\mathbb S\kern 0pt}\nolimits}
\def\ZZ{\mathop{\mathbb Z\kern 0pt}\nolimits}
\def\TT{\mathop{\mathbb T\kern 0pt}\nolimits}
\def\P{\mathop{\mathbb P\kern 0pt}\nolimits}
\def \H{{\mathop {\mathbb H\kern 0pt}\nolimits}}
\newcommand{\ds}{\displaystyle}

\newcommand{\Z}{{\ZZ}}

\newcommand{\beq}{\begin{equation}}
\newcommand{\eeq}{\end{equation}}
\newcommand{\ben}{\begin{eqnarray}}
\newcommand{\een}{\end{eqnarray}}
\newcommand{\beno}{\begin{eqnarray*}}
\newcommand{\eeno}{\end{eqnarray*}}
\newcommand{\bqs}{\begin{equation*}}
\newcommand{\eqs}{\end{equation*}}
\newcommand{\andf}{\quad\hbox{and}\quad}
\newcommand{\with}{\quad\hbox{with}\quad}

\def \cFH {\cF_\H}
\def\equivH#1 {\buildrel\hbox{\tiny {$#1$}}\over \equiv}
\def\simH#1 {\buildrel\hbox{\footnotesize {$#1$}}\over \sim}
\def\prodHO  {\buildrel\hbox{\footnotesize {$\wh \H^d_0$}}\over\cdot}

\newtheorem{definition}{Definition}[section]
\newtheorem{theorem}{Theorem}[section]
\newtheorem{lemma}{Lemma}[section]
\newtheorem{remark}{Remark}[section]

\newtheorem{proposition}{Proposition}[section]
\numberwithin{equation}{section}

\begin{document}
\title[Fourier transform  on the Heisenberg group]
{Tempered distributions and Fourier transform on the Heisenberg group}
 \author[H. BAHOURI] {Hajer Bahouri} 
 \address[H. Bahouri]%
{LAMA, UMR 8050\\
Universit\'e Paris-Est Cr\'eteil, 94010 Cr\'eteil Cedex, FRANCE}
  \email{hajer.bahouri@math.cnrs.fr}
  \email{}    
 \author[J.-Y. CHEMIN]{Jean-Yves Chemin}
\address [J.-Y. Chemin]%
{Laboratoire J.-L. Lions, UMR 7598 \\
Universit\'e Pierre et Marie Curie, 75230 Paris Cedex 05, FRANCE }
\email{chemin@ann.jussieu.fr}
\author[R. DANCHIN]{Raphael Danchin }%
\address[R. Danchin] {LAMA, UMR 8050\\
Universit\'e Paris-Est Cr\'eteil, 94010 Cr\'eteil Cedex, FRANCE}
  \email{danchin@u-pec.fr}
 
\date{\today}

\begin{abstract} The final goal of the present work is to  extend the Fourier  transform on the Heisenberg group $\H^d,$ to tempered distributions. 
 As in  the Euclidean setting, the strategy is to first show that the Fourier transform is an isomorphism on the Schwartz space, 
then to define the extension by duality. 
The difficulty that is here encountered  is that the Fourier transform of an integrable function on $\H^d$
is no  longer a function on $\H^d$ : according to the standard definition, 
it is a family of bounded operators on  $L^2(\R^d).$ 
Following our new approach in\ccite{bcdFHspace}, we here define the  Fourier transform of an integrable function
to be a  mapping   on the  set~$\wt\H^d=\N^d\times\N^d\times\R\setminus\{0\}$
endowed with a suitable distance $\wh d$.
This viewpoint turns out to  provide a user friendly description of   
the range of the Schwartz space on $\H^d$ by the Fourier transform, which makes the extension 
to the whole set of  tempered distributions straightforward.   As a first application, we give an explicit 
formula for the Fourier transform of smooth functions on $\H^d$ that are independent of the vertical variable. 
We also provide other  examples.  
\end{abstract}

\maketitle


\vskip 0.2cm


\setcounter{equation}{0}
\section{Introduction}
\label {intro} 
The present work aims at extending  Fourier  analysis on the Heisenberg group from integrable functions to tempered distributions.  It is by now very classical that in the case of a commutative group, the  Fourier transform  is a function on  the group of characters. In the   Euclidean space $\R^n$  the  group of characters may be identified to the dual space~$(\R^n)^\star$ of~$\R^n$  through the map~$\xi \mapsto  e^{i\langle \xi ,\cdot\rangle},$ where $\langle \xi, \cdot \rangle$  designate the value of the one-form $\xi$ when applied to elements of ~$\R^n$, and the Fourier transform
 of an integrable function~$f$  may  be seen as a function on $(\R^n)^\star$,  defined by the formula
\beq
\label {definFourierclassic}
\cF (f) (\xi) = \wh f(\xi)\eqdefa \int_{\R^n} e^{-i\langle  \xi,x\rangle } f(x)\, dx.
\eeq

A fundamental fact of the distribution theory on $\R^n$ is  that  the Fourier transform
  is a bi-continuous isomorphism on  the  Schwartz space~$\cS(\R^n)$  -- the set of smooth  functions  whose derivatives decay at  infinity  faster than any power. 
  Hence, one can  define the transposed Fourier transform~${}^t\!\cF$ on 
  the   so-called set of tempered distributions~$\cS'(\R^n),$ 
 that is   the topological dual of~$\cS(\R^n)$  (see e.g. \cite{rudin, S} for  a self-contained presentation).
      Now,  as the whole distribution theory on~$\R^n$  is based on  identifying locally integrable functions  
  with  linear forms by means of the Lebesgue integral, it is natural to look for 
a more direct relationship between~${}^t\!\cF$ and~$\cF,$   by considering the following bilinear form on~$\cS(\R^n)\times\cS(\R^n)$
\beq
\label {definFourierRbilin}
\cB_\R (f,\phi)\eqdefa \int_{T^\star \R^n} f(x) e^{-i\langle \xi,x\rangle} \phi(\xi) \,dx \,d\xi,
\eeq
where the cotangent bundle~$T^\star \R^n$ of~$\R^n$  is  identified to~$\R^n\times (\R^n)^\star$. 
The  above bilinear form allows to identify ${}^t\!\cF_{|\cS((\R^n)^\star)}$   to 
 $\cF_{|\cS(\R^n)},$ and  still makes sense if~$f$ and~$\f$ are in~$L^1(\R^n)$,  because the  function~$f\otimes \f$ is integrable on~$T^\star \R^n$. 
 It is thus natural to define the extension of $\cF$ on $\cS'(\R^n)$ to be ${}^t\!\cF.$ 
 In other words, 
\beq
\label {definFourierS'}
  \forall (T,\phi) \in  \cS'(\R^n)\times\cS(\R^n)\,,\  \langle \wh T,\phi\rangle_{\cS'(\R^n)\times\cS(\R^n)}\eqdefa \langle T,\wh\phi\rangle_{\cS'(\R^n)\times\cS(\R^n)}.
  \eeq
  
  We aim at  implementing  that  procedure  on the  Heisenberg group~$\H^d.$ As in the Euclidean case, to achieve our goal, 
 it is fundamental to have a  handy characterization  of 
the range of the Schwartz space on $\H^d$ by the Fourier transform. 
The first   attempt   in that direction goes back to the pioneering works by Geller in \cite{geller2, geller}
(see also \ccite{astengo2,F,thangavelu2} and the references therein), 
where   asymptotic series are used.   
Whether the  description  of $\cF(\cS(\H^d))$ given therein allows to extend the Fourier transform to tempered distribution is unclear, though. 
 \bigbreak
 Before presenting our main results, we have to recall the definitions of the Heisenberg group $\H^d$ and of the Fourier transform on~$\H^d.$ 
Throughout this paper we shall see~$\H^d$ 
  as the  set~$T^\star\R^d \times\R$  equipped with the product law
  $$
w\cdot w'\eqdefa
\bigl(Y+Y' , s+s'+ 2\s(Y,Y')\bigr) = \bigl(y+y',  \eta+\eta' , s+s'+2 \langle \eta,y'\rangle -2\langle \eta',y\rangle\bigr)
$$
where~$w=(Y,s)=(y,\eta,s)$ and $w'=(Y',s')=(y',\eta',s')$ are  generic elements of~$\H^d.$ In the above definition, the notation~$\langle \cdot,\cdot\rangle$ 
 designates the duality bracket  between~$(\R^d)^\star $ and~$\R^d$ and~$\s$ is the canonical symplectic form on~$\R^{2d}$ seen as~$T^\star\R^d$. This gives on~$\H^d$ a structure of a non commutative group for which~$w^{-1}=-w$.  
We refer for instance to the books   \cite{bfg, farautharzallah, fisher, folland, follandstein, stein2, taylor1,  thangavelu} and the references therein for further details.

\smallbreak
In accordance with the above product formula, one can define the set of the dilations on the Heisenberg group 
to be the family of operators $(\d_a)_{a>0}$ given by
\beq\label {defdilations}\d_a(w) = \d_a (Y,s) \eqdefa (aY,a^2s). \eeq
Note that dilations commute with the product law  on $\H^d,$  that is $\d_a(w\cdot w') = \d_a(w) \cdot \d_a (w').$  Furthermore, as 
   the determinant of~$\d_a$ (seen as an automorphism  of $\R^{2d+1}$) is~$a^{2d+2},$  it is natural to define the \emph{homogeneous dimension} of~$\H^d$ to be~$N\eqdefa2d+2$.
\smallbreak
The Heisenberg group is endowed with a smooth left invariant  \emph{Haar measure,} which,  in the coordinate system~$(y, \eta, s)$ is just  the  Lebesgue measure on $\R^{2d+1}.$  The corresponding Lebesgue spaces $L^p (\H^d)$  are thus the sets of measurable functions $f:\H^d\to\C$ such that 
$$
 \|f\|_{L^p (\H^d)}  \eqdefa \left( \int_{\H^d} |f(w)|^p \: dw \right)^\frac1p<\infty, \quad\hbox{if }\ 1\leq p<\infty,
 $$
 with the standard modification if~$p=\infty.$
 \medbreak
 The convolution product  of any two integrable functions~$f$ and~$g$ is given by
\beq
\label {definConvolH}
f \star g ( w ) \eqdefa \int_{\H^d} f ( w \cdot v^{-1} ) g( v)\, dv 
= \int_{\H^d} f ( v ) g( v^{-1} \cdot w)\, dv.
\eeq
As in the Euclidean case, the convolution product is an \emph{associative}
binary operation on the set of integrable functions. Even though it is no longer commutative, 
 the following \emph{Young inequalities}   hold true:
\[ \norm{ f \star g }{L^r} \leq \norm f {L^p} \norm g {L^q},
 \quad \hbox{whenever}\ 1\leq p,q,r\leq\infty\ \hbox{ and }\
 \frac{1}{r} =  \frac{1}{p} + \frac{1}{q} - 1.
\]

The \emph{Schwartz space} $\cS(\H^d)$ corresponds to the
  Schwartz space~$\cS(\R^{2d+1})$ (an equivalent definition involving the Heisenberg structure will be provided in  Appendix \ref{subellipticity}).

As the Heisenberg group  is  noncommutative, it is unfortunately not possible 
to define the Fourier transform of integrable functions on $\H^d,$
by a formula similar to\refeq{definFourierclassic}, just resorting to the characters of $\H^d.$ Actually, the group of characters on~$\H^d$ is isometric to the group  of characters  on~$T^\star\R^d$ and,  if one defines the Fourier transform 
 according to Formula\refeq{definFourierclassic}  then  the information pertaining to the vertical variable~$s$ is  lost. 
One has to use a more elaborate family of irreducible  representations.  As explained for instance in \cite{taylor1} Chapter 2, all irreducible  representations of $\H^d$ are unitary equivalent to  
the \emph{Schr\"odinger representation}~$(U^\lam)_{\lambda\in\R\setminus\{0\}}$  which is  the  family of  
group homomorphisms $w\mapsto U^\lam_w$
between~$\H^d$ and the  unitary group~$\cU(L^2(\R^d))$ of~$L^2(\R^d)$
defined for all~$w=(y,\eta, s)$ in~$\H^d$ and $u$ in $L^2(\R^d)$ by 
$$
U^\lam _w u(x)\eqdefa e^{-i\lam (s+2\langle \eta, x-y\rangle)} u(x-2y).
$$ 
The standard definition of the Fourier transform reads as follows.
\begin{definition}
\label {definFourierSchrodinger}
{\sl For $f$ in $L^1(\H^d)$ and $\lam$ in~$ \R\setminus\{0\}$, we define 
$$
\cF^{\H} (f)(\lam) \eqdefa \int_{\H^d} f(w) U^\lam_w\, dw.
$$
The function $\cF^{\H} (f)$ which takes values in the space of bounded operators on $L^2(\R^d)$, is by definition the \emph{ Fourier transform}   \index{Fourier!transform}
of  $f$.
}
\end{definition}

As  the map~$w\mapsto U^\lam_w$ is a homomorphism
between~$\H^d$ and the  unitary group~$\cU(L^2(\R^d))$ of~$L^2(\R^d)$, it is clear that for any couple $(f,g)$ of integrable functions, we have
\beq
\label {FourierConvol} 
\cF^{\H} (f\star g) (\lam)  =  \cF^{\H} (f)(\lam)\circ \cF^{\H}(g)(\lam).
\eeq

An obvious drawback of Definition\refer{definFourierSchrodinger} is that $\cF^\H f$ \emph{is not} 
a complex valued function on some `frequency space', but a much more complicated
object. Consequently,  with this viewpoint, one can hardly expect to have    a characterization of the range of
the Schwartz space by~$\cF^\H,$ allowing for our  extending  the Fourier transform to tempered distributions.

To overcome that difficulty, we proposed in our recent paper\ccite{bcdFHspace}
 an alternative (equivalent) definition 
that makes the Fourier transform of any  integrable  function on $\H^d,$ a continuous function 
on another (explicit and simple) set~$\wh\H^d$ endowed with some distance  $\wh d$. 

Before giving our definition, we need to introduce some notation. 
Let us first recall that the Lie algebra  of   \emph{left invariant} vector fields,
that is vector fields commuting with any   left translation~$\tau_w(w') \eqdefa w\cdot w'$,
  is spanned by the vector fields
$$ S\eqdefa\partial_s\,,\ \  \cX_j\eqdefa\partial_{y_j} +2\eta_j\partial_s\andf \Xi_j\eqdefa \partial_{\eta_j} -2y_j\partial_s\,,\ 1\leq j\leq d.
$$
The  \emph{Laplacian} \index{Laplacian}
associated to the vector fields~$(\cX_j)_{1\leq j\leq d}$ and~$(\Xi_j)_{1\leq j\leq d}$ is defined by
\begin{equation}
\label{defLaplace}
\D_{\H}  \eqdefa \sum_{j=1} ^d (\cX_j^2+\Xi_j^2), 
\end{equation}
and may be alternately rewritten in terms of  the usual derivatives as follows:
\beq
\label {laplacienexplicite}
\D_\H f(Y,s)  = \D_{Y} f (Y,s) + 4\sum_{j=1} ^d  ( \eta_j \partial_{y_j} -y_j\partial_{\eta_j}) \partial_s f(Y,s) +4|Y|^2\partial_s^2 f(Y,s).
\eeq

The Laplacian plays a fundamental role in the Heisenberg group and in particular in 
the Fourier transform theory. 
The starting point is the following relation 
that holds true for functions on the Schwartz space (see e.g.\ccite{huet, O}):
\beq
\label {FourierEtLaplace}
\cF^{\H}(\D_\H f) (\lam) = 4\cF^{\H}(f)(\lam)  \circ  \D_{\rm osc} ^\lam \with \D_{\rm osc}^\lam u (x) \eqdefa\sum_{j=1}^d \partial_j^2 u(x) - \lam^2|x|^2 u(x).
\eeq
In order to  take advantage of  the spectral structure of the harmonic oscillator, 
it is natural to introduce the corresponding eigenvectors, that is 
the family of Hermite functions~$(H_n)_{n\in\N^d}$
defined by 
\beq
\label{Hermite functions}
H_n \eqdefa  \Bigl(\frac 1 {2^{|n|} n!}\Bigr) ^{\frac 12}C^n H_0 \with C^n \eqdefa \prod_{j=1}^d C_j^{n_j}
\andf H_0(x)\eqdefa \pi^{-\frac d 2} e^{-\frac {|x|^2} 2},
\eeq
where~$C_j\eqdefa -\partial_j +M_j$ stands for  the 
 \emph{creation operator} with respect to the $j$-th variable
 and~$M_j$ is the multiplication operator  defined by~$M_ju(x)\eqdefa x_ju(x).$
 As usual, $n!\eqdefa n_1!\dotsm n_d!$ and $|n|\eqdefa n_1+\cdots+n_d$. 
\medbreak
Recall that~$ \suite H n {\N^d}$ is an orthonormal basis of~$L^2(\R^d),$ and 
that we have
\beq
\label {relationsHHermite}
( -\partial_j^2+M_j^2) H_n =( 2n_j+1) H_n \quad\hbox{and thus}\quad -\D_{\rm osc}^1 H_n = (2|n|+d) H_n.
\eeq
For~$\lam$ in~$\R\setminus\{0\},$ we finally introduce
the rescaled Hermite function~$H_{n,\lam} (x)\eqdefa |\lam|^{\frac d 4} H_n(|\lam|^{\frac 12} x)$. 
It is obvious that~$(H_{n,\lam})_{n\in \N^d}$ is still  an orthonormal basis of~$L^2(\R^d)$ and that 
\beq
\label {relationsHHermiteD}
( -\partial_j^2+\lam^2M_j^2) H_{n,\lam} =( 2n_j+1)|\lam| H_{n,\lam} \quad\hbox{and thus}\quad -\D_{\rm osc}^\lam H_{n,\lam} = (2|n|+d)|\lam| H_{n,\lam}.
\eeq
\begin{remark}
{\sl The vector fields
$$
\wt \cX_j \eqdefa \partial_{y_j} -2\eta_j \partial_s \andf \wt \Xi_j \eqdefa \partial_{\eta_j} +2y_j \partial_s 
$$
are right invariant and we have
$$
\cF^\H (\wt \D_\H f) (\lam) = 4\D_{\rm osc} ^\lam \circ \cF^\H(f)(\lam).
$$
}
\end{remark}
Our alternative definition of the Fourier transform on $\H^d$ reads as follows:
\begin{definition}
\label {definFouriercoeffH}
{\sl 
Let~$\wt \H ^d\eqdefa \N^{2d}\times \R\setminus\{0\}.$ We denote by~$\wh w=(n,m,\lam)$ a generic point of~$\wt \H^d$. For~$f$ in~$L^1(\H^d)$,  we define 
the map~$\cFH f$ (also denoted by~$\wh f_\H$) to be 
$$
\cFH f: \ 
\left\{
\begin{array}{ccl}
\wt \H ^d  & \longrightarrow & \C\\[1ex]
\wh w & \longmapsto & 
\bigl(\cF^{\H}(f)(\lam) H_{m,\lam} |H_{n,\lam}\bigr)_{L^2}.
\end{array}
\right.
$$
}
\end{definition}

To underline the similarity between that definition and the classical one in $\R^n,$ 
one may further compute $\bigl(\cF^{\H}(f)(\lam) H_{m,\lam} |H_{n,\lam}\bigr)_{L^2}.$
One can observe that, after  an obvious change of variable, the Fourier transform recasts  in terms of 
the mean value of~$f$ \emph{modulated by  some oscillatory functions}
which are closely related to Wigner transforms of Hermite functions,  namely
\begin{eqnarray}\label {definFourierWigner}
\cF_\H f(\wh w) &&\!\!\!\!\!\!\!= \int_{\H^d}  \overline{e^{is\lam} \cW(\wh w,Y)}\, f(Y,s) \,dY\,ds
\with\\\label{definWigner}
\ds \cW(\wh w,Y) && \!\!\!\!\!\!\!\! \eqdefa \int_{\R^d} e^{2i\lam\langle \eta,z\rangle} H_{n,\lam} (y+z) H_{m,\lam} (-y+z) \,dz.
\end{eqnarray}

Let us emphasize that with  this new point of view,   Formula\refeq {FourierEtLaplace} recasts as follows:
\beq
\label {FourierdiagDeltaHfond}
\cF_{\H}(\D_\H f) (\wh w) = -4|\lam|(2|m|+d) \wh f_\H(\wh w).
\eeq

Furthermore, if we  endow the set~$\wt \H^d$ with the measure $d\wh w$ defined  by the relation
\beq
\label {definmeasurewhH}
\int_{\wt \H^d} \theta (\wh w)\, d\wh w\eqdefa \sum_{(n,m)\in \N^{2d}} \int_{\R} \theta (n,m,\lam) |\lam|^d d\lam,
\eeq
then  the classical inversion formula
and Fourier-Plancherel  theorem recast as follows:
\begin{theorem}\label {inverseFourier-Plancherel}
{\sl 
Let~$f$ be a function in~$\cS(\H^d).$
 Then  we have   the inversion formula
 \beq
\label {MappingofPHdemoeq1}
f(w) = \frac {2^{d-1}}  {\pi^{d+1} }   \int_{\wt \H^d} 
e^{is\lam} \cW(\wh w, Y)\wh f_\H (\wh w) \, d\wh w\qquad\hbox{for any}\ w \hbox{ in}\ \H^d.
\eeq
Moreover, the Fourier transform~$\cF_{\H}$ can be extended  into a bicontinuous isomorphism between~$L^2(\H^d)$ and~$L^2(\wt \H^d),$ which satisfies
\beq
\label {inverseFouriereq2}
\|\wh f_\H\|_{L^2(\wt \H^d)}^2 = \frac {\pi^{d+1}} {2^{d-1}} \|f\|_{L^2(\H^d)}^2.
\eeq
Finally, for any couple $(f,g)$ of integrable functions, the following convolution identity holds true:
\begin{equation}
\begin {aligned}
\label {newFourierconvoleq1}
 &\cFH (f\star g) (n,m,\lam)   = ( \wh f_\H \cdot \wh g_\H)(n,m,\lam) \with\\ 
 & \qquad\qquad( \wh f_\H  \cdot \wh g_\H)(n,m,\lam)  \eqdefa \sum_{\ell\in \N^{d}} \wh f_\H(n,\ell,\lam)\wh g_\H(\ell,m,\lam).
 \end {aligned} 
 \end{equation}
 }
\end{theorem}
For the reader's  convenience, we present a proof of Theorem\refer{inverseFourier-Plancherel} in the appendix.


\section {Main results} \label {main}

As already mentioned, our main goal  is to extend  the Fourier transform to tempered distributions on~$\H^d$. 
If we follow the standard approach of the Euclidean setting, that is described by\refeq {definFourierRbilin} and\refeq {definFourierS'}, 
then we need  a handy description of the range of~$\cS(\H^d)$ by the Fourier transform~$\cF_\H$
in order to guess what could be the appropriate bilinear form $\cB_\H$ allowing for identifying ${}^t\!\cF_\H$ with $\cF_\H.$
To characterize $\cF(\cS(\H^d)),$  we shall just keep in mind  the most obvious properties we expect the Fourier transform to have. 
The first one is that it should  change  regularity of functions on $\H^d$ to  decay of the Fourier transform. 
This is achieved in the following lemma (see the proof in \ccite{bcdFHspace}).
\begin{lemma}
\label {decaylambdan}
{\sl 
For any integer~$p$, there exist an integer~$N_p$ and a positive constant $C_p$~such that 
for all~$\wh w$ in~$\wt \H^d$ and all~$f$ in~$\cS(\H^d),$ we have 
\begin{equation}
\label {eq:decay}
\bigl(1+ |\lam|(  |n| + |m|+ d) +|n-m|  \bigr)^p  |\wh f_\H(n,m,\lam)|  \leq  C_p  \|f\|_{N_p,\cS},
\end{equation}
where~$\|\cdot\|_{N,\cS}$ denotes the classical family of semi-norms of~$\cS(\R^{2d+1})$, namely
$$\|f\|_{N,\cS}\eqdefa \sup_{ |\al|\leq N}  \bigl\|(1+|Y|^2+s^2)^{N/2}\,\partial_{Y,s}^\al f \bigr\|_{L^\infty}.$$}
\end{lemma}
The   decay inequality \eqref{eq:decay} prompts us 
 to  endow the set $\wt\H^d$ with the following distance~$\wh d$:
\beq
\label {defindistancewtH}
\wh d(\wh w,\wh w') \eqdefa \bigl|\lam(n+m)-\lam'(n'+m')\bigr|_1 +\bigl |(n-m)-(n'-m')|_1+|\lam-\lam'|,
\eeq
where~$|\cdot|_1$ denotes the~$\ell^1$ norm on~$\R^d$. 
\medbreak
The second basic property we expect for  the Fourier transform  is that it changes decay properties into regularity. This is closely related to  how it acts on suitable weight functions. As in the Euclidean case, we expect $\cF_\H$ to transform multiplication by  weight functions into a   combination of derivatives,
so  we need  a   definition   of differentiation  for functions defined on $\wt\H^d$
that could fit the  scope. This  is the aim of  the following definition (see also Proposition \ref{actionX_jonFH} in Appendix): 
\begin{definition} 
{\sl  For any function~$\theta:\wt\H^d\to\C$ we define
 \begin{multline}
  \label {decayWignerHermiteeq1}
  \wh \D \theta(\wh w) \eqdefa - \frac 1{2|\lam|} ( |n+m| +d) \theta(\wh w) \\
+\frac 1 {2|\lam|} \sum_{j=1} ^d \Bigl ( \sqrt {(n_j+1) (m_j+1)}\, \theta(\wh w_j^+) 
+\sqrt {n_jm_j}\, \theta(\wh w_j^-)\Bigr)
 \end{multline}
and, if in addition $\theta$ is differentiable with respect to $\lambda,$ 
 \begin{equation}\label {eq:whDlambda}
\wh\cD_\lam \theta(\wh w)  \eqdefa  \frac {d\theta} {d\lam} (\wh w)+ \frac d {2\lam} \theta(\wh w)\\
+\frac 1{2 \lam}\sum_{j=1}^d \Bigl(
\sqrt {n_jm_{j} }\,  \theta(\wh w_j^-)-\sqrt {(n_j\!+\!1)(m_j\!+\!1)} \, \theta(\wh w_j^+) \Bigr)
\end{equation}
where~$\wh w_j^\pm\eqdefa(n\pm\delta_j,m\pm\delta_j,\lam)$ and~$\d_j$ denotes the element of~$\N^d$  with all  components equal to~$0$ except the~$j$-th which has value~$1$.}
\end{definition}
The notation in the above definition  is justified by the following lemma 
that will be proved in  Subsection \ref{Regularity}.
 \begin{lemma}
 \label {decaygivesregul1}
{\sl Let~$M^2$ and~$M_0$ be the multiplication operators defined on $\cS(\H^d)$ by
 \begin{equation}
 \label {weight}
(M^2f)(Y,s)\eqdefa|Y|^2f(Y,s)\quad\hbox{and}\quad
M_0f(Y,s)\eqdefa-isf(Y,s).
 \end{equation}
Then for all~$f$  in~$\cS(\H^d),$  the following two 
relations hold true on $\wt\H^d$:
$$
\cF_{\H} M^2 f=  -\wh \D\cF_\H f\quad\hbox{and}\quad 
 \cF_\H(M_0 f)  =   \wh\cD_\lam \cF_\H f.
 $$
 }
 \end{lemma}
 
 The third important  aspect of regularity for functions in $\cF_\H(\cS(\H^d))$  is the link between their  values for positive $\lam$ and negative~$\lam$.
 That property, that has no equivalent in the Euclidean setting,  
  is described by the following lemma:
 \begin{lemma}
\label {symmetry0} 
{\sl Let us consider on ~$\cS(\H^d)$    the operator~$\cP$ defined by
\begin{equation}\label{eq:cP}
\cP(f) (Y,s) \eqdefa \frac 12 \int_{-\infty} ^s \bigl(f(Y,s') -f(Y,-s')\bigr) ds'.
\end{equation}
Then   $ \cP$ maps continuously~$\cS(\H^d)$ to~$\cS(\H^d)$ and we have for any~$f$ in~$\cS(\H^d)$ and~$\wh w$ in~$\wt\H^d,$
\begin{equation}\label{eq:lam=0}
2 i \cF_\H (\cP f) = \wh\Sigma_0 (\cF_\H f ) \with 
(\wh\Sigma_0 \theta) (\wh w) \eqdefa \frac {\theta (n,m,\lam) -(-1) ^{|n+m|} \theta (m,n,-\lam)} \lam\,\cdotp
\end{equation}
}
\end{lemma} 
The above weird relation  is just a consequence of  the following  property of the Wigner transform $\cW$: \begin{equation}\label{eq:symW}
\forall (n,m,\lam,Y) \in \wt \H^d\times T^\star \R^d\,,\   \cW(n,m,\lam,Y)=(-1)^{|n+m|}\cW(m,n,-\lam,Y).
\end{equation}
In the case $m=n,$ it means that 
the left and right limits at~$\lam=0$ of functions in $\cF_\H(\cS(\H^d))$ must be  the same.

\begin{definition}
\label {definrangeS}
{\sl
We define~$\cS(\wt \H ^{d})$ to be the set of  functions~$\theta$ on~$\wt \H^d$ such that:
\begin{itemize}
\item 
for any~$(n,m)$ in~$\N^{2d}$, the map~Ê$\lam\longmapsto \theta(n,m,\lam)$ is  smooth  on~$\R\setminus\{0\}$,
\item
for any non negative integer~$N$, the functions~$\wh \D^N\theta$,~$\wh\cD_\lam^N\theta$ and~$\wh \Sigma_0\wh\cD_\lam^N\theta$  decay faster than any power of~$\wh d_0(\wh w)\eqdefa |\lam| (|n+m|_1 +d) +|m-n|_1$.
\end{itemize}
We equip $\cS(\wt\H^d)$  with the family of semi-norms
$$
\|\theta \|_{N,N'\!,\cS(\wt \H^d)} \eqdefa \sup_{\wh w\in \wh\H^d} \bigl( 1+\wh d_0(\wh w)\bigr) ^N
\Bigl(  |\wh \D^{N'}  \theta (\wh w)|+  |\wh \cD_\lam^{N'} \theta (\wh w)| + |\wh \Sigma_0\wh \cD_\lam^{N'}\theta (\wh w)|\Bigr)\cdotp
$$
}
\end{definition}

Let us first point out that  an integer~$K$ exists such that 
\begin{equation}
\label{control}
\|\theta\|_{L^1(\wt\H^d)}\leq  C\|\theta\|_{K,0,\cS(\wt\H^d)}.
\end{equation}

The main motivation of this definition is the following isomorphism theorem.
\begin{theorem}
\label {MappingofPH}
{\sl 
The Fourier transform $\cF_\H$ is a bicontinuous isomorphism between~$\cS(\H^{d})$ and~$\cS(\wt \H^d),$ 
and the inverse map is given by
\beq\label {MappingofPHdemoeq1b}
\wt\cF_\H \theta(w)\eqdefa \frac {2^{d-1}}  {\pi^{d+1} }   \int_{\wt \H^d} 
e^{is\lam} \cW(\wh w, Y)\theta (\wh w) \, d\wh w.
\eeq}
\end{theorem}

The definition of $\cS(\wt\H^d)$ encodes a number of  nontrivial hidden informations  that are partly consequences of the sub-ellipticity of $\D_\H.$ For instance, the stability of~$\cS(\wt \H^d)$ by the multiplication law defined in\refeq {newFourierconvoleq1} is an obvious consequence of the stability of~$\cS(\H^d)$ by convolution and of Theorem\refer {MappingofPH}. Another hidden information is the behavior of functions of~$\cS(\wt \H^d)$ when~$\lam$ tends to~$0$. 
In  fact, Achille's heel  of the metric space~$(\wt\H^d,\wh d)$ is that it is not complete. 
It turns out however that  the Fourier transform of any integrable function on~$\H^d$
is  \emph{uniformly continuous} on  $\wt\H^d.$ Therefore,  it is natural to 
extend it to the  completion~$\wh\H^d$ of~$\wt\H^d.$ This is explained in greater details
in the following statement that has been proved in\ccite {bcdFHspace}. 
\begin{theorem}
\label {FourierL1basicbis}
{\sl 
The completion of~$(\wt \H^d, \wh d)$ is the metric space~$(\wh \H^d,\wh d)$ defined by
$$
\wh \H^d\eqdefa  \wt\H^d \cup \wh \H^d_0 \with \wh \H^d_0 \eqdefa {\R_{\mp}^d}\times \Z^d
\andf
{\R_{\mp}^d}\eqdefa (\R_-)^d\cup (\R_+)^d.
$$
Moreover, on $\wh\H^d,$ the extended distance (still denoted by~$\wh d$) is given 
for all $\wh w=(n,m,\lam)$ and~$\wh w'=(n',m',\lam')$ in $\wt\H^d,$
and for all $(\dot x,k)$ and $(\dot x',k')$ in $\wh\H_0^d$ by
$$
\begin{aligned}
&\wh d(\wh w,\wh w') = \bigl|\lam(n+m)-\lam'(n'+m')\bigr|_1 +\bigl |(m-n)-(m'-n')|_1+|\lam-\lam'|,\\
&\wh d\bigl (\wh w, (\dot x, k)\bigr ) =  \wh d\bigl ((\dot x, k), \wh w \bigr )  \eqdefa |\lam(n+m)-\dot x|_1+ |m-n-k|_1+|\lam|, \\
&\wh d\bigl ((\dot x,k), (\dot x', k')\bigr )  = |\dot x-\dot x'|_1+|k-k'|_1.
\end{aligned}
$$

The Fourier transform $\wh f_\H$ of any integrable function on $\H^d$ may be extended continuously 
to the whole set $\wh\H^d.$ Still denoting by  $\wh f_\H$ (or $\cF_\H f$) that extension, 
the  linear map~$\cF_\H: f\mapsto \wh f_\H$ is continuous from  the space $L^1(\H^d)$  to the space 
$\cC_0(\wh\H^d)$  of continuous functions  on~$\wh \H^d$ tending to~$0$ at infinity. 
}
\end{theorem}

It is now natural to introduce the space~$\cS(\wh\H^d)$.
 \begin{definition}
\label  {definrangeS+} 
{\sl
We denote by~$\cS(\wh \H^d)$ the space of functions on~$\wh\H^d$  which
are continuous extensions of elements of~$\cS(\wt \H^d)$.
}
 \end{definition}

As  an  elementary exercise of functional analysis, the reader can  prove that~$\cS(\wh \H^d)$
endowed with the semi-norms $\|\cdot\|_{N,N'\!,\cS(\wt \H^d)}$ is a Fr\'echet space.
Those semi-norms will be denoted by~$\|\cdot\|_{N,N'\!,\cS(\wh \H^d)}$ in all that follows.

Note also that for any function $\theta$ in $\cS(\wh\H^d),$  having  $\wh w$ tend to $(\dot x,k)$ 
in \eqref{eq:lam=0} yields 
\begin{equation}\label{eq:pmlambda}
\theta(\dot x,k)=(-1)^{|k|}\theta(-\dot x,-k).
\end{equation}

As regards convolution, we   obtain, after  passing to the limit in \eqref {newFourierconvoleq1}, the following noteworthy   formula, valid for any  two functions~$f$ and~$g$ in~$L^1(\H^d)$:
\beq  \label{convlimit}
\begin{aligned}
\cF_\H(f\star g)_{\wh\H^d_0} &= (\cF_\H f)_{\wh\H^d_0}\prodHO  (\cF_\H g)_{\wh\H^d_0}\with\\
(\theta_1\prodHO  \theta_2)(\dot x,k)  & \eqdefa \sum_{k'\in\Z^d} \theta_1(\dot x,k')\, \theta_2(\dot x,k-k').
\end{aligned}
\eeq

\begin{remark}\label{rk: convlim} {\sl
Let us emphasize  that the above   product law\refeq{convlimit} is commutative even though convolution of
functions on the Heisenberg group \emph{is not} (see\refeq {newFourierconvoleq1}).}
\end{remark}

A natural question then is how to extend the measure $d\wh w$ to $\wh\H^d.$ In fact, we have
 for any positive real numbers~$R$ and~$\e$,
\beno
\int_{\wh \H^d} {\bf 1}_{\{|\lam|\,|n+m|+|m-n|\leq R\}}  {\bf 1} _{|\lam|\leq \e} \,d\wh w &  = & 
\int_{-\e} ^\e \Bigl(\sum_{n,m}    {\bf 1}_{\{|\lam|\,|n+m|+|m-n|\leq R\}}\Bigr) |\lam|^d \,d\lam\\
 &\leq & CR^{2d}\int_{-\e} ^\e\, d\lam\\
 &\leq & CR^{2d}\e.
\eeno
Therefore, one  can extend the measure~$d\wh w$ on~$\wh \H^d$ simply by defining, for any continuous compactly supported function~$\theta$ on~$\wh\H^d$
$$
\int_{\wh \H^d}  \theta (\wh w)\,d\wh w \eqdefa  \int_{\wt \H^d}  \theta (\wh w)\,d\wh w.
$$

At this stage of the paper, pointing out nontrivial examples of functions of~$\cS(\wh \H^d)$
is highly informative.  
 To this end, we  introduce the set~$\cS^+_d$  of smooth functions~$f$ on~$[0,\infty[^d\times \ZZ^d\times\R$ such that for  any integer~$p$, we have
\begin{equation}\label {definprofileSwhH+} 
\supetage{(x_1,\cdots,x_d,k,\lam) \in [0,\infty[^d\times\ZZ^d\times\R}{|\al|\leq p} \bigl( 1+ x_1+\cdots +x_d+|k|\bigr)^p
\bigl |\partial^\alpha_{x,\lam} f (x_1,\cdots,x_d,k,\lam)\bigr|<\infty.
\end{equation}
As  may be easily checked by the reader,  the space~$\cS^+_d$ is stable by derivation and   multiplication by polynomial
functions of~$(x,k)$.
\begin{theorem}
\label {radialtypeSdescribtion}
{\sl 
Let~$f$ be a function of~$\cS^+_d$.   Let us define for~$\wh w=(n,m,\lam)$ in~$\wt \H^d$,
$$
\Theta_f\bigl(\wh w)\eqdefa f\bigl (|\lam|R(n,m) ,m-n,\lam\bigr)\with R(n,m)\eqdefa (n_j+m_j+1)_{1\leq j\leq d}.
$$
 Then~$\Theta_f$ belongs to~$\cS(\wh\H^d)$ if 
 \begin{itemize}
 \item either~$f$ is supported in~$[0,\infty[^d\times \{0\}\times\R,$
 \item or~$f$ is supported in~$[r_0,\infty[^d\times \ZZ^d\times \R$ for some positive real number~$r_0,$ and satisfies
 \begin{equation}\label{eq:fsym}
 f(x, -k, \lam) = (-1)^{|k|}f(x, k, \lam).\end{equation} 
 \end{itemize}  
}
\end{theorem}
An obvious consequence of   Theorem\refer {radialtypeSdescribtion} is  that the fundamental solution of the heat equation in~$\H^d$  belongs to~$\cS(\H^d)$  
(a highly nontrivial result that is usually   deduced from the explicit formula established by B. Gaveau in\ccite{gaveau}). 
Indeed, applying the Fourier transform with respect to the Heisenberg variable
 gives  that if~$u$ is the solution of the heat equation with integrable initial data~$u_0$ then
\begin{equation}\label{eq:heat}
\wh u_\H(t,n,m,\lam) = e^{-4t|\lam|(2|m|+d)} \wh u_0(n,m,\lam).
\end{equation}
At the same time, we have 
$$
u(t) =  u_0\star h_t \with h_t(y,\eta,s)= \frac{1}{t^{d+1}} h\Bigl(\frac{y}{\sqrt{t}}, \frac{\eta}{\sqrt{t}}, \frac s t\Bigr).
$$
Hence combining the convolution formula\refeq {newFourierconvoleq1} 
and  Identity \eqref{eq:heat}, we gather  that
$$
\wh h_\H (\wh w) = e^{-4|\lam|(2|n|+d)} {\bf 1}_{\{n=m\}}.
$$ 
Then  applying Theorem\refer {radialtypeSdescribtion}  to the function~$e^{-4(x_1+\cdots+x_d)}{\bf 1}_{\{k=0\}}$ 
ensures  that~$\wh h_\H$ belongs to~$\cS(\wh\H^d),$ and the inversion theorem \ref {MappingofPH} thus implies that 
$h$ is in $\cS(\H^d).$
\smallbreak

Along the same lines, we recover  Hulanicki's theorem
\cite{hula} in the case of the Heisenberg group, namely if $a$ belongs to $\cS(\R)$, then there exists a  function~$h_a$ in~$ {\mathcal S}(\H^d)$ such that
\begin{equation}\label{hultheorem}
 \forall f\in {\mathcal S}(\H^d),\;\; a(-\D_{\H})f=f*h_a .
\end{equation}

\medbreak
As already explained in the introduction, our final aim is to  extend  the Fourier  transform   to tempered distributions by adapting the Euclidean procedure described
in \eqref {definFourierRbilin}--(\ref{definFourierS'}). The purpose of the following definition is to specify 
what a \emph{tempered distribution}  on $\wh\H^d$ is.
\begin{definition}
\label {definSprimewhH}
{\sl    Tempered distributions    on~$\wh\H^d$ are  elements of the set~$\cS'(\wh \H^d)$  of continuous linear forms on the Fr\'echet space~$\cS(\wh \H^d)$. 
  
  We say that a sequence~$\suite \cT n \N$ of tempered distributions on~$\wh\H^d$  converges to a tempered distribution~$\cT$ if 
$$
\forall \theta \in \cS(\wh \H^d)\,, \ \lim_{n\rightarrow \infty} \langle \cT_n,\theta\rangle _{\cS'(\wh\H^d)\times \cS(\wh \H^d)} = \langle \cT,\theta\rangle _{\cS'(\wh\H^d)\times \cS(\wh \H^d)} .
$$}
\end{definition}
Let us now give some examples of elements of~$\cS'(\wh \H^d)$ and present the most basic properties of this space. As a start, let us  specify what are functions with moderate growth.
\begin{definition}
\label {definL1Moderated}
{\sl Let us denote by~$ L^1_M(\wh\H^d) $ the space of locally integrable functions $f$ on 
$\wh\H^d$ such that  there exists an integer~$p$ satisfying
$$
\int_{\wh \H^d} (1+|\lam|(n+m|+d)+|n-m|)^{-p}|f(\wh w) |d\wh w<\infty.
$$ 
}
\end{definition}
As in the Euclidean setting,  functions of~$L^1_M(\wh \H^d)$  may be identified to tempered distributions:
\begin{theorem}
\label {indentifonctdistritemp}
{\sl Let us  consider~$\iota$ be the map defined by
$$
\iota :\left\{
\begin{array}{ccl}
 L^1_M(\wh\H^d) & \longrightarrow & \cS'(\wh\H^d) \\
 \psi & \longmapsto & \displaystyle \iota(\psi)\,: \Bigl[\theta  \mapsto  \int_{\wh \H^d} \psi(\wh w)\theta(\wh w)\, d\wh w\Bigr]\cdotp
\end{array}
\right.
$$
Then $\iota$ is a one-to-one linear map. 
\medbreak 
 Moreover,  if~$p$ is an integer such that the map
 $$(n,m,\lambda)\longmapsto(1+|\lam|(n+m|+d)+|n-m|)^{-p}f(n,m,\lam)
 $$ belongs to~$L^1(\wh \H^d),$ then
 we have
\begin{equation}\label{eq:boundf}
|\langle \iota(f),\f\rangle | \leq \bigl\|(1+|\lam|(n+m|+d)+|n-m|)^{-p}f\bigr\|_{L^1(\wh \H^d)}
\|\theta\|_{p,0,\cS(\wh\H^d)}.
\end{equation}}
\end{theorem}
The following proposition provides examples of  functions in ~$L^1_M(\wh \H^d)$.
 \begin{proposition}
\label  {dimhomoconcretwhH}
{\sl 
For any $\gamma<d+1$  the function~$f_\g$ defined on~$\wh\H^d$ by
$$
f_\g(n,m,\lam) \eqdefa \bigl(|\lam| (2|m|+d)\bigr) ^{-\g} \, \delta_{n,m}
$$
belongs to~$L^1_M(\wh \H^d)$.
}
\end{proposition}
\begin{remark}
{\sl The above proposition is no longer true for~$\g=d+1$. If we look at  the quantity~$|\lam| (2|n|+d)$ in $\wh\H^d$ as an equivalent of~$|\xi|^2$ for~$\R^d$, then 
it  means  that the \emph{homogeneous dimension of~$\wh \H^d$} is~$2d+2,$  as for~$\H^d\,$ (and as expected).}
\end{remark}

It is obvious that any Dirac mass on~$\wh\H^d$ is a tempered distribution. Let us also note that 
because
$$
|\theta (n,n,\lam) | \leq \bigl( 1+|\lam|( 2 |n |+d)\bigr)^{-d-3} \|\theta\|_{d+3,0, \cS(\wh \H^d)},
$$
the linear form
\beq
\label{IdentitydiscreteFourier}
\cI : \left\{
\begin{array}{ccl}
\cS(\wh\H^d) & \longrightarrow &  \C\\
\theta & \longmapsto & \ds \sum_{n\in \N^d} \int_{\R} \theta (n,n,\lam)\,|\lam|^dd\lam
\end{array}
\right.
\eeq
is a tempered distribution on~$\wh\H^d$.
\medbreak
We now  want to exhibit tempered distributions on~$\wh \H^d$ which \emph{are not} measures. 
 The following proposition states that  the analogue  on $\wh\H^d$  of  \emph{finite part}  distributions on~$\R^n,$
 are indeed in $\cS'(\wh\H^d).$ 
\begin{proposition}
\label {examplepartiefinie}
{\sl 
Let~$\g$ be in the interval~$]d+1,d+3/2[$ and denote by
$\wh 0$ the element $(0,0)$ of $\wh\H^d_0.$ 
Then for any function~$\theta$ in~$\cS(\wh\H^d),$ the function
defined a.e. on $\wh\H^d$ by 
$$
(n,m,\lam)\longmapsto \d_{n,m}\biggl(\frac {\theta(n,n,\lam)+ \theta(n,n,-\lam) -2\theta(\wh 0) }{ |\lam|^\g (2|n|+d)^\g}\biggr), 
$$
is integrable. Furthermore, the linear form defined by
$$
\Big\langle {\rm Pf} \Bigl(\frac 1 {|\lam|^\g(2|n|+d)^\g}\Bigr),  \theta \Big \rangle \eqdefa \frac 12 \int_{\wh \H^d} 
\biggl(\frac {\theta(n,n,\lam)+ \theta(n,n,-\lam) -2\theta(\wh 0) }{ |\lam|^\g (2|n|+d)^\g} \biggr)\d_{n,m} \,d\wh w
$$
is in~$\cS'(\wh\H^d),$ and its restriction to~$\wt\H^d$ is the function 
$$
(n,m,\lam) \longmapsto \d_{n,m}  \frac 1 {|\lam|^\g(2|n|+d)^\g} 
$$
in the sense that for any~$\theta $ in~$\cS(\wh\H^d)$  such that $\theta(n,n,\lambda)=0$ for
small enough $|\lambda|(2|n|+d),$  we have
$$
\Big\langle {\rm Pf}\Bigl( \frac 1 {|\lam|^\g(2|n|+d)^\g}\Bigr),  \theta \Big \rangle 
= \int_{\wh \H^d}   \frac {\theta (\wh w)}  {|\lam|^\g(2|n|+d)^\g}\, d\wh w.
$$
}
\end{proposition}

Another  interesting example of tempered distribution  on~$\wh\H^d$  is the measure $\mu_{\wh \H^d_0}$ 
defined in Lemma~3.1 of\ccite{bcdFHspace} which, in our setting, recasts as follows:
\begin{proposition}
\label {ex}
{\sl Let the measure~$\mu_{\wh \H^d_0}$ be defined by
\beq
\label {limitmeasureeq1}
\langle \mu_{\wh \H^d_0} ,\theta \rangle = \int_{\wh \H^d_0}  \theta (\dot x,k) \,d\mu_{\wh \H^d_0}(\dot x,k)
\eqdefa2^{-d}\sum_{k\in \Z^d}  \biggl( \int_{(\R_{-})^d} \theta(\dot x,k)\,d\dot x +  \int_{(\R_{+})^d} \theta(\dot x,k)\,d\dot x\biggr)
\eeq
for all functions $\theta$ in   $\cS(\wh\H^d).$
\medbreak
Then  $\mu_{\wh \H^d_0}$ is a tempered distribution on~$\wh\H^d$ and for any function~$\psi$ in~$\cS(\R)$ with integral~$1$ we have  
$$
\lim_{\e\to0}   \frac 1 {\e}  \psi\Big(\frac \lam {\e}\Bigr)  = \mu_{\wh \H^d_0} \quad\hbox{in} \ \cS'(\wh \H^d).
$$ 
}
\end{proposition}

Let us finally explain how  the Fourier transform may be extended  to  tempered distributions on~$\H^d,$ using an analog of Formulas\refeq {definFourierRbilin}
 and\refeq {definFourierS'}. Let us define 
 \ben
 \label {definFourierHbilin}
&& \cB_\H:
 \left\{
 \begin{array}{ccl}
 \cS(\H^d)\times \cS(\wh \H^d) & \longrightarrow & \C\\
 (f,\theta) & \longmapsto  & \ds \int_{\H^d\times \wh \H^d} f(Y,s) \, \overline{e^{is\lam} \cW(\wh w,Y)}\, \theta(\wh w) \,dwd\wh w
 \end{array}
 \right.\andf\\
 \label {definFourierHtranspose}
 && {}^t\cF_\H:  \left\{  \begin{array}{ccl}
 \cS(\wh \H^d) & \longrightarrow & \cS(\H^d) \\
 \theta & \longmapsto  & \ds \int_{\wh \H^d} \overline{e^{is\lam} \cW(\wh w,Y)}\, \theta(\wh w) \,d\wh w.
 \end{array}
 \right.
 \een
Let us notice that  for any $\theta$ in $\cS(\wh\H^d)$ and $w=(y,\eta,s)$ in $\H^d,$ we have
\beq
\label {tFequivF-1}
({}^t \cF _\H \theta)( y,\eta, s) = \frac {\pi^{d+1} } {2^{d-1} } (\cF_\H^{-1}  \theta)( y,-\eta, -s) .
\eeq
Hence, Theorem\refer {MappingofPH} implies that~${}^t\!\cF_\H$ is a continuous isomorphism between~$\cS(\wh\H^d)$ and~$\cS(\H^d)$.  Now, we observe
that for any $f$ in $\cS(\H^d)$ and $\theta$ in $\cS(\wh\H^d),$ we have
\beq
\label {bilineatandtranspose}
\cB_\H(f,\theta) = \int_{\H^d} f(w) ({}^t \!\cF _\H \theta)( w)\,dw =  \int_{\wh \H^d} (\cF _\H f)(\wh w) \theta(\wh w)\,d \wh w. 
\eeq
This prompts us to extend $\cF_\H$ on $\cS'(\H^d)$ as follows: 
\begin{definition}
{\sl We define
$$
\cF_\H: \left\{ 
\begin{array}{ccl}
\cS'(\H^d)& \longrightarrow & \cS'(\wh \H^d)\\
 T & \longmapsto & \Bigl[\theta \mapsto \langle T, {}^t\cF_\H \theta\rangle_{\cS'(\H^d)\times \cS(\H^d)}\Bigr]\cdotp
\end{array}
\right.
$$}
\end{definition}
As a direct consequence of this definition, we have the following statement:
\begin{proposition}
\label {FS'Hconitnuous}
{\sl 
The map~$\cF_{\H}$ defined just above is continuous and  one-to-one  from~$\cS'(\H^d)$ onto~$\cS'(\wh\H^d).$ Furthermore, its
restriction to~$L^1(\H^d)$ coincides with Definition\refer{definFouriercoeffH}.
}
\end{proposition}Just to compare with the Euclidean case, let us give some examples of simple computations of Fourier transform of
tempered distributions on $\H^d$.
\begin{proposition}
\label {Fourierdetaand1}
{\sl We have
$$
\cF_\H (\delta_0) = \cI\andf \cF_\H({\bf 1}) = \frac {\pi^{d+1}} { 2^{d-1}} \delta_{\wh 0},
$$
where $\cI$ is defined by \eqref{IdentitydiscreteFourier} and  $\wh 0$ is the element of 
$\wh \H^d_0$ corresponding to $\dot x=0$ and $k=0$.}
\end{proposition}

One question that comes up naturally is to compute the Fourier transform of a function \emph{independent of the vertical variable}.  
The answer to that question is  given just below.  \begin{theorem}
\label  {Fourierhorizontal+}
{\sl
We have for any integrable function~$g$ on~$T^\star\R^d$,
$$
\cF_\H(g\otimes {\bf1} ) = (\cG_\H g) \mu_{\wh \H^d_0}
$$
where ~$\cG_\H g $ is defined by
\beno
\cG_\H g: &&\!\!\!\!\!\!\left\{
\begin{array} {rcl}
\wh  \H^d_0 & \longrightarrow & \C\\
 (\dot x, k) & \longmapsto & \ds  \int_{T^\star \R^d}  \ov\cK_d(  \dot x ,k,Y) g(Y)\, dY
\end{array} 
\right. \with\\
 \cK_d (\dot x, k,Y) & = & \bigotimes_{j=1}^d \cK (\dot x_j, k_j,Y_j)\andf \\
\cK(\dot x, k,y,\eta) &\eqdefa &  
\frac1{2\pi} \!\int_{-\pi}^{\pi} \! e^{i \left(2|\dot x|^{\frac 12} (y\sin z + \eta \sgn(\dot x) \cos z) +kz\right)}\, dz.
\eeno
}
\end{theorem}

As we shall see, this result  is just an interpretation of Theorem 1.4 of\ccite{bcdFHspace} in terms of tempered distributions.

\medbreak
The rest of the paper unfolds as follows.
  In Section\refer  {rangeofcS},  we prove Lemmas\refer  {decaygivesregul1} and\refer {symmetry0}, and then Theorem\refer {MappingofPH}.
In Section\refer {ExamplescSwhH}, we establish Theorem\refer {radialtypeSdescribtion}.
In Section\refer {examplescS'},  we study in full details the examples of tempered distributions on~$\wh\H^d$ given in  Propositions\refer  {dimhomoconcretwhH}--\ref  {examplepartiefinie}, and Theorem\refer   {indentifonctdistritemp}. 
In Section\refer  {computeFHcS'}, we prove Proposition\refer {Fourierdetaand1}
and Theorem\refer {Fourierhorizontal+}. Further  remarks  as well as  proofs  (within our setting) of known results 
are postponed in the appendix.


\section {The range of the Schwartz class by the Fourier transform}
\label {rangeofcS}

The present section aims at giving  a handy characterization of the range of $\cS(\H^d)$ 
by the  Fourier transform.
Our Ariadne thread throughout will be that we expect that, for the action of~$\cF_\H,$  
 \emph{regularity implies decay}
and  \emph{decay implies regularity}. The answer to the first issue has been given in  
 Lemma\refer{decaylambdan} (proved in \cite {bcdFHspace}). Here we shall  concentrate on the second issue, 
 in connection with the definition of differentiation 
for functions on $\wt\H^d,$ given in  \eqref{decayWignerHermiteeq1} and \eqref {eq:whDlambda}.
To complete our analysis of the space $\cF_\H(\cS(\H^d)),$
we will have to get some information on the behavior of elements of 
$\cF_\H(\cS(\H^d))$ for $\lambda$ going to $0$ (that is in the neighborhood 
of the set $\wh\H^d_0$).  This is Lemma  \ref{symmetry0} that points out  an extra 
and fundamental relationship  between \emph{positive} and \emph{negative}~$\lam$'s. 
\smallbreak
A great deal of our program will be achieved 
by describing  the action of the weight function $M^2$ and of the differentiation 
operator $\partial_\lam$ on $\cW.$ This is the goal of the next paragraph.

\subsection{Some properties for Wigner transform of Hermite functions}

The following lemma describes the action of the weight function $M^2$ on $\cW.$
\begin{lemma}
 \label {Y2WignerHermite}
{\sl
For all~$\wh w$ in~$\wt\H^d$ and $Y$ in~$ T^\star\R^d,$ we have 
 $$|Y|^2\cW(\wh w,Y)  = -\wh\D \cW(\cdot ,Y) (\wh w) $$ 
 where Operator $\wh\D$ has been defined in\refeq{decayWignerHermiteeq1}.
}
\end{lemma}
\begin{proof}
{}From the definition of~$\cW$ and  integrations by parts, we get  
\beno
|Y|^2\cW(\wh w,Y) & =  &\int_{\R^d} \Bigl(|y|^2-\frac 1 {4\lam^2} \D_z\Bigr) \bigl(e^{2i\lambda\langle \eta,z\rangle}\bigr)  H_{n,\lam} (y+z) H_{m,\lam} (-y+z)\, dz \\
&= & \int_{\R^d} e^{2i\lambda\langle \eta,z\rangle} |\lam|^{\frac d2} \cI(\wh w,y,z) \,dz \\\with
 \cI(\wh w,y,z)   &\eqdefa &  \Bigl(|y|^2-\frac 1 {4\lam^2} \D_z\Bigr)\bigl (H_{n} (|\lam|^{\frac 12} (y+z))  H_{m} (|\lam|^{\frac 12} (-y+z))\bigr).
\eeno
{}From  Leibniz formula, the chain rule and the following identity:
$$
4|y|^2=|y+z|^2+|y-z|^2+2(y+z)\cdot(y-z),
$$
 we get 
\beno
\cI(\wh w,y,z) & = &  -\frac 1 {4\lam^2} \bigl( (\D_z-\lam^2 |y+z|^2) H_{n} (|\lam|^{\frac 12} (y+z)) \bigr)  H_{m} (|\lam|^{\frac 12} (-y+z)) \\
&&{}
-\frac 1 {4\lam^2} \bigl( (\D_z-\lam^2 |y-z|^2) H_{m} (|\lam|^{\frac 12} (-y+z)) \bigr)  H_{n} (|\lam|^{\frac 12} (y+z)) \\
&&{}-\frac 1 {2|\lam|}\sum_{j=1}^d(\partial_j H_{n}) (|\lam|^{\frac 12} (y+z)) (\partial_j H_{m}) (|\lam|^{\frac 12} (-y+z))\\
&&{}-\frac 12 (z+y)\cdot(z-y) H_{n} (|\lam|^{\frac 12} (y+z))  H_{m} (|\lam|^{\frac 12} (-y+z)).
\eeno
Using\refeq{relationsHHermiteD}, we end  up with
\beno
\cI(\wh w,y,z) &= & \frac 1{2|\lam|} (|n+m|+d) H_{n} (|\lam|^{\frac 12} (y+z))  H_{m} (|\lam|^{\frac 12} (-y+z))
\\
&&\qquad \qquad{}
-\frac 1{2|\lam|} \sum_{j=1}^d\Bigl\{
(\partial_j H_{n}) (|\lam|^{\frac 12} (y+z)) (\partial_j H_{m}) (|\lam|^{\frac 12} (-y+z))\\
&&\qquad \qquad\qquad \qquad\qquad \qquad{}
+(M_j H_{n}) (|\lam|^{\frac 12} (y+z)) (M_j H_{m}) (|\lam|^{\frac 12} (-y+z))\Bigr\}\cdotp
\eeno
Then, taking advantage of\refeq{relationsHHermiteCAb}, 
 we get Identity\refeq  {decayWignerHermiteeq1}.
\end{proof}

\medbreak

The purpose of the following lemma is to investigate the action of   $\partial_\lam$ on $\cW$.
  \begin{lemma}
 \label {decayWignerHermite}
{\sl
 We have, for all~$\wh w$ in~$\wt\H^d$,  the following formula: 
 \beq
\label {decayWignerHermiteeq2}
\begin{aligned}
& \partial_\lam \cW(\wh w,Y)   = -\frac d{2\lam}  \cW(\wh w,Y)\\
 &\qquad\qquad\qquad{}+\frac 1 {2\lam} \sum_{j=1} ^d \Bigl\{ \sqrt {(n_j+1) (m_j+1)}\, \cW(\wh w^+_j,Y) - \sqrt {n_jm_j}\, 
 \cW(\wh w^-_j,Y)\Bigr\}\,\cdotp
 \end{aligned}
 \eeq
 }
 \end{lemma}
 \begin{proof} Let us write that
$$
\begin{aligned}
\partial_\lam \cW(\wh w,Y) &=  \int_{\R^d} \frac d {d\lam} \Bigl( |\lam|^{\frac d 2} e^{2i\lam\langle \eta,z\rangle} H_n (|\lam|^{\frac 12} (y+z)) H_m(|\lam|^{\frac 12}  (-y+z)) \Bigr) dz\\
& =  \frac d {2\lam} \cW(\wh w,Y)  +\cW_1(\wh w,Y) + \cW_2(\wh w,Y)  \with\\
\cW_1(\wh w,Y) &\eqdefa  
 \int_{\R^d} 2i\langle \eta,z\rangle e^{2i\lam\langle \eta,z\rangle}  |\lam|^{\frac d 2}  H_n (|\lam|^{\frac 12} (y+z)) H_m(|\lam|^{\frac 12}  (-y+z))\, dz\ \hbox{and}\\
\cW_2(\wh w,Y) &\eqdefa 
\int_{\R^d}  e^{2i\lam\langle \eta,z\rangle}  |\lam|^{\frac d 2}  \frac d {d\lam} \bigl(H_n (|\lam|^{\frac 12} (y+z)) H_m(|\lam|^{\frac 12}  (-y+z))\bigr)\, dz.
\end{aligned}
$$
As we have 
$$
  2i\langle \eta,z\rangle e^{2i\lam\langle \eta,z\rangle } = \frac 1 \lam \sum_{j=1}^d z_j\partial_{z_j} 
e^{2i\lam\langle \eta,z\rangle } ,
$$
an integration by parts gives
\begin{multline}
 \label {decayWignerHermitedemoeq1}
\cW_1(\wh w,Y)  = -\frac d {\lam}  \cW(\wh w,Y)\\
-\frac1\lambda \sum_{j=1} ^d \int_{\R^d} 
e^{2i\lam\langle \eta,z\rangle}  |\lam|^{\frac d 2}  z_j\partial_{z_j}  \bigl(H_n (|\lam|^{\frac 12} (y+z)) H_m(|\lam|^{\frac 12}  (-y+z))\bigr) dz.
\end{multline}
Now let us compute 
$$
\cJ(\wh w,y,z) \eqdefa \Bigl (\frac d {d\lam} -\frac1\lambda\sum_{j=1}^d z_j\partial_{z_j}\Bigr)  \bigl(H_n (|\lam|^{\frac 12} (y+z)) H_m(|\lam|^{\frac 12}  (-y+z))\bigr).
$$
{}From the chain rule we get
\beno
\cJ(\wh w,y,z)  &= & \frac {|\lam|^{\frac 12} }{2\lam} \sum_{j=1}^d \Bigl\{
 (y_j+z_j) H_m(|\lam|^{\frac 12}  (-y+z)) (\partial_j H_n) (|\lam|^{\frac 12} (y+z)) \\
&&\qquad\qquad{}
+(-y_j+z_j)  H_n (|\lam|^{\frac 12} (y+z))  (\partial_j H_m)(|\lam|^{\frac 12}  (-y+z))\\
&&\qquad\qquad{}
- 2 z_j  H_m(|\lam|^{\frac 12}  (-y+z)) (\partial_jH_n) (|\lam|^{\frac 12} (y+z))\\
&&\qquad\qquad\qquad{}
-2 z_j H_n (|\lam|^{\frac 12} (y+z))  (\partial_j H_m)(|\lam|^{\frac 12}  (-y+z))
\Bigr\}\cdotp
\eeno
This gives 
$$
\longformule{
\cJ(\wh w,y,z)  =  - \frac {1 }{2\lam} \sum_{j=1}^d \Bigl\{
(\partial_j H_n) (|\lam|^{\frac 12} (y+z)) |\lam|^{\frac 12}( -y_j+z_j) H_m(|\lam|^{\frac 12}  (-y+z))
}
{
{}
+ |\lam|^{\frac 12} (y_j+z_j)H_n (|\lam|^{\frac 12} (y+z))  (\partial_j H_m)(|\lam|^{\frac 12}  (-y+z))\Bigl\}
}
$$
which writes
$$
\longformule{
\cJ(\wh w,y,z)  =  - \frac {1 }{2\lam} \sum_{j=1}^d \Bigl\{
(\partial_j H_n) (|\lam|^{\frac 12} (y+z))(M_jH_m)(|\lam|^{\frac 12}  (-y+z))
}
{
{}
+ (M_jH_n)(|\lam|^{\frac 12} (y+z))  (\partial_j H_m)(|\lam|^{\frac 12}  (-y+z))\Bigl\}\cdotp
}
$$
Using  Relations\refeq {relationsHHermiteCAb} completes the proof of the Lemma.
 \end{proof}


\subsection{Decay  provides regularity} \label{Regularity} 

Granted with Lemmas\refer {Y2WignerHermite} and\refer{decayWignerHermite}, it is now easy
to establish Lemma\refer{decaygivesregul1}.
Indeed,  according to\refeq {definFourierWigner}, we have
$$
(\cF_\H M^2f)(\wh w)  =   \int_{\H^d} e^{-is\lam}  f(Y,s) |Y|^2  \ov\cW(\wh w,Y) \,dY\,ds.
$$
Therefore, Lemma \refer{Y2WignerHermite}  implies that 
$$\displaylines{\quad
(\cF_\H M^2f)(\wh w)   = \frac 1{2|\lam|} ( |n+m| +d)  \int_{\H^d} f(Y,s)e^{-is\lam} \ov\cW(\wh w,Y)\,dY\,ds\hfill\cr\hfill
-\frac 1 {2|\lam|} \sum_{j=1} ^d \Bigl\{ \sqrt {(n_j+1) (m_j+1)}  \int_{\H^d}  f(Y,s) e^{-is\lam}
 \ov\cW(\wh w_j^+,Y)\,dY\,ds\hfill\cr\hfill
+\sqrt {n_jm_j} \int_{\H^d}  f(Y,s) e^{-is\lam}  \ov\cW(\wh w_j^-, Y) \,dY\,ds \Bigr\}\cdotp\quad}$$
By the definition of the Fourier transform and of $\wh\D,$ this gives $\cF_{\H} M^2 f=  -\wh \D\cF_\H f.$
\medbreak
To establish\refeq{eq:whDlambda}, we start from \refeq{definFourierWigner} and get
$$\begin{aligned}
 \cF_\H(M_0 f)(\wh w) & =   \int_{\H^d} \frac d {d\lam} \bigl(e^{-is\lam} \bigr)  f(Y,s) \ov\cW(\wh w,Y) \,dY\,ds\\
 & =  \frac d {d\lam}   (\cF_\H f)(\wh w)
 -  \int_{\H^d} e^{-is\lam}   f(Y,s)   \frac d {d\lam} \bigl( \ov\cW(\wh w,Y)  \bigr) \,dY\,ds.
\end{aligned}
$$
Rewriting the last term according to  Formula\refeq {decayWignerHermiteeq2}, we discover that 
$$\displaylines{\quad
(\cF_\H M_0f)(\wh w)   =    \frac d {d\lam}   (\cF_\H f)(\wh w) 
+\frac d{2\lam}  \int_{\H^d} f(Y,s)e^{-is\lam} \ov\cW(\wh w,Y)\,dY\,ds\hfill\cr\hfill
-\frac 1 {2\lam} \sum_{j=1} ^d \Bigl\{ \sqrt {(n_j+1) (m_j+1)}  \int_{\H^d}  f(Y,s) e^{-is\lam}
 \ov\cW(\wh w_j^+,Y) \,dY\,ds\hfill\cr\hfill
-\sqrt {n_jm_j} \int_{\H^d}  f(Y,s) e^{-is\lam}  \ov\cW(\wh w_j^-, Y)\, dY\,ds \Bigr\}\cdotp\quad}
$$
By the definition of the Fourier transform,  this  concludes the proof of
Lemma\refer{decaygivesregul1} .
 \qed
 
\medbreak

On the one hand, Lemmas\refer{decaylambdan} and\refer{decaygivesregul1} guarantee 
 that decay in the physical space provides regularity in the Fourier space, 
and that regularity gives decay.  On the other hand,  the relations we established so 
far do not give much insight  on the behavior of the Fourier transform near $\wh\H_0^d$
even though we know from   Theorem\refer{FourierL1basicbis} that
in the case of an integrable function, it has to be  uniformly continuous \emph{up to $\lambda=0$}.
Getting more information  on the behavior of 
the Fourier transform of functions in $\cS(\H^d)$ in a neighborhood of $\wh\H_0^d$ 
is what we want to do now with the proof of Lemma\refer {symmetry0}. 
\begin{proof}[Proof of Lemma\refer {symmetry0}]
Fix some function $f$ in $\cS(\H^d),$ and observe that
$$
\partial_s \cP f (Y,s) =  \frac 12 \bigl(f(Y,s) -f(Y,-s)\bigr)\quad\hbox{with } \cP \ \hbox{defined in \eqref{eq:cP}}.
$$
Taking the Fourier transform with respect to the variable~$s$ gives
\begin{equation}\label{eq:FsP}
i\lam \cF_s (\cP f )(Y,\lam) =  \frac 12 \bigl(\cF_sf(Y,\lam) -\cF_sf(Y,-\lam)\bigr).
\end{equation}
Let us consider a function~$\chi$ in~$\cD(\R)$ with value~Ê$1$ near~$0$ and let us write 
$$
i \cF_s (\cP f )(Y,\lam) = \frac {1-\chi(\lam)} {2\lam} \bigl(\cF_sf(Y,\lam) -\cF_sf(Y,-\lam)\bigr)
+\chi(\lam) \int_0^1 (\partial_\lam\cF_sf) (Y, -\lam+2t\lam) dt.
$$
It is obvious that the two terms in the right-hand side belong to~$\cS(\R^{2d+1})$.  Thus the operator
$$
\f \longmapsto  \frac { \f(Y,\lam) -\f(Y,-\lam)} {2\lam} 
$$
maps continuously~$\cS(\R^{2d+1})$ to~$\cS(\R^{2d+1}).$
Hence $\cP$ maps continuously~$\cS(\H^d)$ to~$\cS(\H^d).$
\medbreak
Note that in the case of a function $g$ in $\cS(\H^d)$, Formula\refeq{definFourierWigner} may be alternately written:
\begin{equation}
\cF_\H g(\wh w)=\int_{T^\star\R^d} \cF_sg(Y,\lam)\ov\cW(\wh w,Y)\,dY\quad\hbox{for all }\ \wh w=(n,m,\lam)
\hbox{ in }\ \wt\H^d.
\end{equation}
Relations \eqref{eq:symW} and \eqref{eq:FsP} guarantee that
$$\begin{aligned}
2 i\lam \cF_\H (\cP f )(\wh w) &= \int_{T^\star\R^d}2i\lam \cF_s(\cP f)(Y,\lam)\ov\cW(\wh w,Y)\,dY\\
&= \int_{T^\star\R^d}\bigl(\cF_sf(Y,\lam)-\cF_sf(Y,-\lam)\bigr)\ov\cW(\wh w,Y)\,dY\\
&= \int_{T^\star\R^d}\cF_sf(Y,\lam)\ov\cW(\wh w,Y)\,dY\\
&\qquad\qquad\qquad\qquad\qquad  {}
-(-1)^{|n+m|} \!\int_{T^\star\R^d}\!\cF_sf(Y,-\lam)\ov\cW(m,n,-\lam, Y)\,dY\\
&=\cF_\H f(n,m,\lam)- (-1)^{|n+m|} \cF_\H f (m,n,-\lam),
\end{aligned}$$
which completes the proof of Lemma\refer{symmetry0}.
\end{proof}


\subsection{Proof of the inversion theorem in the Schwartz space}
The aim of this section is to prove Theorem \ref {MappingofPH}.
To this end, let us first note that from Inequality\refeq{eq:decay} and Lemmas\refer{decaygivesregul1} and \refer{symmetry0}, we gather that~$\cF_\H$ maps $\cS(\H^d)$ to $\cS(\wh\H^d).$
In addition,  \refeq{control} guarantees that all elements of $\cS(\wh\H^d)$ are 
in $L^1(\wh\H^d)\cap L^2(\wh\H^d).$ 

Hence Theorem\refer{inverseFourier-Plancherel}  ensures that $\cF_\H: \cS(\H^d)\to\cS(\wh\H^d)$ is one-to-one, and that  the inverse map has to be the functional $\wt\cF_\H$ defined in\refeq{MappingofPHdemoeq1b}. 
Therefore, there only  remains to prove that~$\wt\cF_\H$ maps $\cS(\wh\H^d)$ to $\cS(\H^d).$ 
To this end, it is convenient to introduce  the following  semi-norms:
 \begin{equation}\label{def:sn}
\|\ f \|_{K,\cS(\H^d)}  \eqdefa \sqrt{\|f\|_{L^2(\H^d)}^2 +\| M_\H^K f\|_{L^2(\H^d)}^2+\|\Delta_\H^K  f\|_{L^2(\H^d)}^2}
\with M_\H\eqdefa M^2+M_0,
\end{equation}
which are equivalent to the classical ones defined in Lemma \ref{decaylambdan} (see Prop. \ref{p:schwartz}).
\medbreak
Let us compute~$M^2\wt\cF_\H \theta (Y,s)$.  According to Lemma\refer{Y2WignerHermite}, we have for all $\wh w=(n,m,\lambda)$ in~$\wt\H^d,$
$$\displaylines{
\sum_{(n,m)\in \N^{2d} } \theta(\wh w) |Y|^2\cW (\wh w,  Y)= 
\frac 1 {2|\lam|}\sum_{(n,m)\in \N^{2d}}\bigg(  
 (|n+m| +d) \cW(\wh w,  Y)\theta(\wh w)\hfill\cr\hfill
- \sum_{j=1}^d \sqrt {n_jm_j} \,  \theta(\wh w) \cW(\wh w_j^-,  Y)
- \sum_{j=1}^d \sqrt {(n_j+1)(m_j+1)} \,\theta (\wh w)  \cW(\wh w_j^+, Y)\biggr)\cdotp}
$$
Changing variable~$(\wt n,\wt m) = (n+\d_j,m+\d_j)$ and~$(\wt n,\wt m) = (n-\d_j,m-\d_j),$ respectively, gives
$$
\begin{aligned}
\sum_{(n,m)\in \N^{2d} } \theta(\wh w) |Y|^2\cW (\wh w, Y)& =  
\frac 1 {2|\lam|} \sum_{(n,m)\in \N^{2d} } \bigg( 
(|n+m| +d) \theta(\wh w) \cW(\wh w,  Y)\\
&\quad - \sum_{j=1}^d \biggl(\sqrt {(n_j\!+\!1)(m_j\!+\!1)} \,\theta (\wh w_j^+)  
+ \sqrt {n_jm_j} \,  \theta(\wh w_j^-)\biggr)\bigg) \cW(\wh w,  Y)\\
&=  -\sum_{(n,m)\in \N^{2d} }  \wh \D \theta (\wh w) \cW(\wh w,  Y)
\end{aligned}
$$
where  $\wh \D$ is the operator introduced in\refeq{decayWignerHermiteeq1}.
\medbreak
Multiplying by $2^{d-1}\pi^{-d-1}e^{is\lambda},$ integrating with respect 
to $\lambda$ and remembering\refeq{MappingofPHdemoeq1b},  we end up with
\beq
\label    {MappingofPHdemoeq3}
(M^2 \wt\cF_\H \theta)(Y,s)  = - \wt\cF_\H (\wh\D\theta)(Y,s).
\eeq
Understanding how~$M_0$ acts on $\wt\cF_\H(\cS(\wh\H^d))$ is more delicate. It requires our using  the continuity
property of Definition\refer {definrangeS}.
Now,  if  $\theta$ is in $\cS(\wh\H^d)$  then it is integrable. As obviously $|\cW|\leq 1,$ one may thus write
for all $w=(Y,s)$ in $\H^d,$ 
    denoting~$\R_\e\eqdefa  \R\setminus[-\e,\e]$,
\beno
(M_0 \wt\cF_\H \theta )(w)  & = & \frac {2^{d-1} } {\pi^{d+1} } \lim_{\e\rightarrow 0} \sum _{(n,m)\in \N^{2d}} 
\Psi_\e(n,m,w) \with\\
\Psi_\e(n,m,w) &\eqdefa & 
- \int_{\R_\e} 
\Bigl( \frac d {d\lam} e^{is\lam}\Bigr)  \theta (n,m,\lam) 
\cW(n,m,\lam, Y) |\lam|^d d\lam\,.
\eeno
Integrating by parts yields
$$
\begin{aligned}
\Psi_\e(n,m,w) & = \Psi^{(1)} _\e(n,m,w) - \Psi^{(2)} _\e(n,m,w) \with\\
\Psi^{(1)} _\e(n,m,w) & \eqdefa     
 \int_{\R_\e} e^{is\lam} \frac d {d\lam} \bigl( \cW(n,m,\lam, Y)\theta (n,m,\lam) |\lam|^d\bigr) d\lam
\andf \\
 \Psi^{(2)} _\e(n,m,w) & \eqdefa   \e^d \bigl( e^{is\e}  \cW (n,m,\e, Y)\theta (n,m,\e)- e^{-is\e}  \cW (n,m,-\e, Y) \theta (n,m, -\e)\bigr).
 \end{aligned}$$
 Let us compute
\beq
\label {MappingofPHdemoeq5}
   \Theta( \wh w, Y)\eqdefa \frac d {d\lam} \bigl( \cW(n,m,\lam,Y)\theta (n,m,\lam) |\lam|^d\bigr).
 \eeq
 Leibniz formula  gives
 $$
   \Theta (\wh w, Y)  =   \partial_\lam \cW(\wh w,Y)\theta(\wh w) |\lam|^d+    \cW(\wh w, Y) \frac d {d\lam} \bigl(  |\lam|^d\theta (\wh w)\bigr) .
 $$ 
 Hence, remembering Identity\refeq  {decayWignerHermiteeq2}, we discover that
$$\displaylines{
\Theta(\wh w,Y)  =   \frac {d\theta} {d\lam}(\wh w) \cW(\wh w,Y)  |\lam|^d + \frac d {2\lam}  \theta(\wh w) \cW (\wh w, Y)  |\lam|^d\hfill\cr\hfill
-\frac{|\lam|^d}{2\lam}\sum_{j=1} ^d   \theta(\wh w)  \Bigl(\sqrt {n_jm_j}\, \cW(\wh w_j^-,\lam, Y) 
 - \sqrt {(n_j+1)(m_j+1)} \cW (\wh w_j^+, Y)\Bigr).}
 $$
{}From the changes of  variable~$(n',m')=(n-\d_j,m-\d_j)$ and~$(\wt n,\wt m)=(n+\d_j,m+\d_j)$, we infer that 
$$
\displaylines {
\sum_{(n,m)\in \N^{2d}}   \theta(\wh w)  \Bigl(\sqrt {n_jm_j} \,\cW(\wh w_j^-, Y) - \sqrt {(n_j\!+\!1)(m_j\!+\!1)}\, \cW (\wh w_j^+, Y)\Bigr) 
\hfill\cr\hfill
=-\!\!\!\sum_{(n,m)\in \N^{2d}} \!\!  \cW(\wh w,Y)   \Bigl(\sqrt {n_jm_j} \,\theta (\wh w_j^-) - \sqrt {(n_j\!+\!1)(m_j\!+\!1)}\, \theta (\wh w_j^+)\Bigr).
}
$$
Therefore, using the operator~$\wh\cD_\lam$ introduced in Lemma\refer{decaygivesregul1}, we get
\beq
\label {MappingofPHdemoeq6}
\sum_{(n,m)\in\N^{2d} } \Psi^{(1)} _\e(n,m,w) =  \sum_{(n,m)\in\N^{2d} } \int_{\R_\e} 
e^{is\lam} (\wh\cD_\lam\theta)(n,m,\lam)\cW (n,m,\lam, Y)  |\lam|^d d\lam.
\eeq
Now let us study the term~$ \Psi^{(2)} _\e(n,m,w)$. We have
\beno
 &&e^{is\e}  \cW (n,m, \e, Y)\theta (n,m, \e)- e^{-is\e}  \cW (n,m,-\e, Y)\theta (n,m,-\e)\\
 &&\qquad\qquad\qquad\qquad{}=
 \bigl( e^{is\e} - e^{-is\e}\bigr) \cW (n,m,\e,  Y)\theta (n,m,\e)\\
&&\qquad\qquad\qquad\qquad\qquad{}
 +e^{-is\e}\bigl(\cW( n,m,\e, Y)\theta (n,m,\e)-  \cW (n,m,-\e,Y)\theta (n,m,-\e) \bigr). 
\eeno
Hence, thanks to\refeq{eq:symW} 
$$
\displaylines{
\sum_{(n,m)\in\N^{2d} }   \Psi^{(2)} _\e(n,m,w)  =  2i \, \e^d \sin (s \e)   \sum_{(n,m)\in\N^{2d} }  \cW (n,m,\e, Y)\theta (n,m,\e)
\hfill\cr\hfill
+ \e^d e^{-is\e}  \Bigl(\sum_{(n,m)\in\N^{2d} }   \cW (n,m,\e, Y)\theta(n,m, \e) 
- \sum_{(n,m)\in\N^{2d} }  (-1)^{|n+m|}\cW (m,n,\e, Y)\theta(n,m,-\e)\Bigr)\cdotp}
$$
Swapping indices~$n$ and $m$  in the last sum gives
$$
\longformule{
\sum_{(n,m)\in\N^{2d} }   \Psi^{(2)} _\e(n,m,w)  =  2i \e^d \sin (s\e)\,   \sum_{(n,m)\in\N^{2d} }  \cW (n,m,\e, Y)\theta (n,m,\e)
}
{{} + \e^{d+1}  e^{-is\e} \sum_{(n,m)\in\N^{2d} }   \cW (n,m,\e, Y)(\wh\Sigma_0\theta)(n,m,\e) .
}
$$
Remembering  that $|\cW| \leq 1,$  we thus get 
\begin{equation}\label{eq:psi2}
\biggl|\sum_{(n,m)\in\N^{2d} }   \Psi^{(2)} _\e(n,m,w)  \biggr|\leq \ep^{d+1}\biggl(
2|s| \sum_{(n,m)\in\N^{2d}}|\theta(n,m,\e)|+ \sum_{(n,m)\in\N^{2d}} |(\wh\Sigma_0\theta)(n,m,\e)|\biggr).  
\end{equation}
Now, let us use the fact that  we have
$$
\sum_{(n,m)\in\N^{2d}}|\theta(n,m,\e)|\leq \|\theta\|_{2d+2,0,\cS(\wt\H^d)}
\sum_{(n,m)\in\N^{2d}} \bigl(1+\ep(|n+m|+d)+|n-m|\bigr)^{-2d-2}.
$$
We observe that
$$
\begin{aligned}
\sum_{(n,m)\in\N^{2d}} \!\!\bigl(1\!+\!\ep(|n\!+\!m|+d)+|n-m|\bigr)^{-2d-2}
& \leq  \sum_{\ell\in\N^{d}} \bigl(1\!+\!\ep(|\ell|+d)\bigr)^{-d-1}
 \sum_{k\in\Z^{d}} \bigl(1+|k|\bigr)^{-d-1}\\
&\leq C  \e^{-d}.
\end{aligned}$$
Hence  the first term of the right-hand side of \eqref{eq:psi2} tends to $0$ when $\e$ goes to $0.$

Employing the same argument  with $\wh\Sigma_0\theta$  
guarantees that the last term  of \eqref{eq:psi2} tends to $0$ when $\e$ goes to $0.$
Therefore, we do have 
$$
\lim_{\e\rightarrow  0} \sum_{(n,m)\in\N^{2d} }  \Psi^{(2)} _\e(n,m,w)  =0.
$$
Using that~$\wh\cD_\lam \theta$ belongs to~$\cS(\wh \H^d)$ and is thus integrable,
 we deduce from\refeq   {MappingofPHdemoeq6} that
$$
\lim_{\e\rightarrow 0}  \sum_{(n,m)\in\N^{2d} } \Psi^{(1)} _\e(n,m,w) =  \int_{\wh \H^d} e^{is\lam}  (\wh\cD_\lam\theta)(\wh w)\cW (\wh w, Y) d\wh w.\\
$$

Thus this gives
\beq
\label {MappingofPHdemoeq7-1}
M_0 \wt\cF_\H \theta = \wt\cF_\H \wh\cD_\lam \theta.
\eeq
Together with\refeq   {MappingofPHdemoeq3}, this implies that
\beq
\label {MappingofPHdemoeq7}
M_\H  \wt\cF_\H \theta =  \wt\cF_\H\bigl( (-\wh \D +\wh\cD_\lam)(\theta)\bigr)\with M_\H\eqdefa M^2+M_0.
\eeq
Hence we can conclude that for any integer $K,$ there exist an integer $N_K$ and a constant  $C_K$ so that 
\begin{equation}\label{eq:MK}
\|M_\H^K  \wt\cF\theta\|_{L^2(\H^d)}\leq C_K\|\theta\|_{N_K,N_K,\cS(\wh\H^d)}.
\end{equation}
Finally, to study the action of the  Laplacian on $\wt\cF_\H(\cS(\wh\H^d),$ we write that by definition of $\cX_j$ and of $\cW,$ we have
\beno
\cX_j \bigl(e^{is\lam} \cW (\wh w, Y) \bigr) & = & \int_{\R^d}  \cX_j \bigl(e^{is\lam +2i\lam \langle \eta,z\rangle} H_{n,\lam}(y+z)
H_{m,\lam} (-y+z)\bigr) dz\\
& = & \int_{\R^d}  e^{is\lam +2i\lam \langle \eta,z\rangle }
\bigl( 2i \lam \eta_j +\partial_{y_j}\bigr) \bigl( H_{n,\lam}(y+z)
H_{m,\lam} (-y+z)\bigr) dz.
\eeno
As~$2i \lam \eta_j e^{2i\lam \langle \eta,z\rangle } = \partial_{z_j}(e^{2i\lam \langle \eta,z\rangle })$,  integrating by parts yields 
\beq
\label {MappingofPHdemoeq8}
\cX_j \bigl(e^{is\lam} \cW (\wh w, Y) \bigr) =  \int_{\R^d}  e^{is\lam +2i\lam \langle \eta,z\rangle }
\bigl( \partial_{y_j}-\partial_{z_j}\bigr) \bigl( H_{n,\lam}(y+z)
H_{m,\lam} (-y+z)\bigr) dz.
\eeq
The action of~$\Xi_j$ is simply described by
\beno
\Xi_j \bigl(e^{is\lam} \cW (\wh w, Y) \bigr)  & = & \int_{\R^d}  \Xi_j \bigl(e^{is\lam +2i\lam \langle \eta,z\rangle}\bigr) H_{n,\lam}(y+z)
H_{m,\lam} (-y+z)\, dz\\
& = & \int_{\R^d}  e^{is\lam +2i\lam \langle \eta,z\rangle}
2i\lam(z_j -y_j)  H_{n,\lam}(y+z)
H_{m,\lam} (-y+z)\, dz.
\eeno
Together with\refeq  {MappingofPHdemoeq8} and the definition of $\D_\H$ in\refeq{defLaplace}, 
this gives
$$\begin{aligned}
\D_\H \bigl(e^{is\lam} \cW(\wh w,Y)\bigr)  &=4\int _{\R^d}  e^{is\lam +2i\lam \langle \eta,z\rangle} 
H_{n,\lam}(y+z)
(\Delta_{osc}^\lambda H_{m,\lam}) (-y+z)\, dz\\&= -4|\lam| (2|m|+d) e^{is\lam} \cW(\wh w,Y).\end{aligned}
$$
This implies  that for all integer $K,$ we have
$$
(-\D_\H)^K (\wt\cF_\H  \theta)  = \wt\cF_\H  \wh M^K \theta \with \wh M \theta (n,m,\lam) \eqdefa 4|\lam|(2|m|+d)\theta (n,m,\lam),
 $$
 whence  there exist an integer $N_k$ and a constant  $C_K$ so that 
\begin{equation}\label{eq:DeltaK}
\|(-\D_\H)^K(\wt\cF_\H\theta)\|_{L^2(\H^d)}\leq C_k\|\theta\|_{N_K,N_K,\cS(\wh\H^d)}.
\end{equation}
Putting \eqref{eq:MK} and \eqref{eq:DeltaK} together and remembering the definition of the semi-norms
on~$\cS(\H^d)$ given in \eqref{def:sn}, we conclude that  for all integer $K,$ there exist an integer $N_K$ and a constant~$C_K$ so that 
$$
\|\wt\cF_\H\theta\|_{K,\cS(\H^d)}\leq C_K\|\theta\|_{N_K,N_K,\cS(\wh\H^d)}.
$$
This completes  the proof of Theorem\refer{MappingofPH}.


\section {Examples of functions in the range of the Schwartz class}
\label {ExamplescSwhH}

The purpose of this section is to prove Theorem \ref {radialtypeSdescribtion}. Let us recall the notation 
$$
\Theta_f\bigl(\wh w)\eqdefa f\bigl (|\lam|R(n,m) ,m-n,\lam\bigr)\with R(n,m)\eqdefa (n_j+m_j+1)_{1\leq j\leq d}.
$$
For any function~$f$ in~$\cS_d^+$ which is either supported in $[0,\infty[^d\times\{0\}\times\R$ or
in $[r_0,\infty[^d\times\Z^d\times\R$ for some positive real number $r_0$ and satisfies \eqref{eq:fsym}, 
 the fact that~$\|\Theta_f\|_{N,0,\cS(\wh \H ^{d})}$ is finite for all integer $N$ is obvious. We  next have to study the action of~$\wh \D$ and~$\wh\cD_\lam$   on~$\Theta_f$.  To this end, we shall  establish a Taylor type expansion of~$\wh \D \Theta_f$ and~$\wh\cD_\lam\Theta_f$ near~$\lam=0$.  To explain what kind of convergence we 
 are looking for, we need the following definition.
 \begin{definition}
\label {definequivonHtilde}
{\sl
Let~$M$ be an integer. We say that two continuous functions~$\theta$ and~$\theta'$ on~$\wh \H^d$ are~$M$-equivalent 
(denoted  by~$\theta\equivH{M} \theta'$)  if 
for all positive integer~$N$, a constant~$C_{N,M}$ exists such that 
$$
\forall \wh w\in \wt \H^d\,,\ |\theta(\wh w)-\theta'(\wh w)| \leq C_{N,M} |\lam|^M (1+|\lam|(|n+m|+d)+|m-n|)^{-N}.
$$
}
\end{definition}
Let us first observe that, if $M\geq1$ then
\beq
\label {equivonHtildeeq0}
\theta\equivH M 0\Longrightarrow \|\theta\|_{N,0,\cS(\wh \H ^{d})}<\infty\ \hbox{ for all integer }N.
\eeq
Furthermore, whenever $0\leq M_0\leq M,$ we have
\beq
\label {equivonHtildeeq1} 
\theta\equivH {M} \theta' \Longrightarrow |\lam|^{-M_0}  \theta\equivH {M-M_0} |\lam|^{-M_0} \theta',
\eeq
and it is obvious that if~$P$ is a function  bounded by a polynomial in~$(n,m)$ with total degree~$M_0$, then
\beq
\label {equivonHtildeeq2} 
\theta\equivH {M} \theta' \Longrightarrow P(n,m)\,\theta\equivH {M-M_0} P(n,m)\,\theta'.
\eeq 
Finally, note  that the definition of~$\wh \D$ in  \eqref{decayWignerHermiteeq1} implies that 
\beq
\label {equivonHtildeeq3}
\theta\equivH {M} \theta' \Longrightarrow \wh \D \theta\equivH {M-2} \wh\D\theta'.
\eeq
We have the following lemma.
\begin{lemma}
\label {TaylorHtilde}
{\sl
For any positive integer~$M$, we have
$$ 
\forall \wh w \in \wt \H^d\,,\ \Theta_f (\wh w_j^\pm)  \equivH { M+{\rm 1}} 
   \sum_{\ell=0} ^{M} \frac {  (\pm2|\lam|)^\ell } {\ell!} \Theta _{\partial_{x_j}^\ell f}(\wh w).
 $$
 }
\end{lemma}

\begin{proof}
Performing a Taylor expansion at order~$M+1$, we get 
$$
\longformule{
f \bigl(|\lam| R(n\pm \d_j, m\pm \d_j),m-n,\lam\bigr)=
   \sum_{\ell=0} ^{M} \frac {(\pm2|\lam|)^\ell } {\ell!}   (\partial_{x_j}^\ell f)(|\lam|R(n,m),m-n,\lam\bigr)
 }
 { {}+ (\pm2|\lam|)^{M+1}  \int_0^1  \frac  {(1-t)^{M}} {M!}  
   (\partial_{x_j}^{M+1} f)\bigl( |\lam|R^\pm_j(n,m,t),m-n,\lam\bigr)\,dt
 }
$$
with $R^\pm_j (n,m,t) \eqdefa \bigl (n_1+m_1+1,\cdots, n_j+m_j+1\pm 2t, \cdots , n_d+m_d+1\bigr)$. The fact that~$f$ belongs to~$\cS_d^+$  implies that  for any positive integer~$N$, we have
$$
\biggl | \int_0^1  \frac  {(1-t)^{M}} {M!}  
   (\partial_{x_j}^{M+1} f)\bigl( |\lam|R^+_j(n,m,t),m-n,\lam\bigr)\,dt \biggr|
   \leq C_N\bigl(1+|\lam|(|n+m|+d)+|m-n|\bigr)^{-N}.
$$
This gives the lemma.
 \end{proof}

\medbreak

One can now tackle  the proof of Theorem\refer{radialtypeSdescribtion}.
Let us first investigate the (easier) case when the support of $f$ is included in~$[0,\infty[^d\times\{0\}\times \R$. The first step consists in computing an equivalent (in the sense of Definition\refer  {definequivonHtilde}) of~$\wh\Delta\Theta_f$ at  an order which will  be chosen later on. For notational simplicity, we here set $R(n)\eqdefa R(n,n)$ and omit the second variable of $f.$
Now, by definition of the operator~$\wh \D$, we have
\beq
\label  {radialSdescribtiondemoeq0}
\begin{aligned}
(-\wh \Delta \Theta_f )(n,\lam) &= \frac 1 {2 |\lam|} \Bigl( (|2n| +d)  f\bigl (|\lam|R(n),\lam\bigr)
-\sum_{j=1}^d  \wt \D_j (n,\lam)\Bigr)\with\\\
 \wt \D_j (n,\lam) & \eqdefa 
(n_j+1)f \bigl(|\lam|R(n+2\d_j),\lam\bigr) 
 +
n_j f \bigl(|\lam|R(n-2\d_j),\lam\bigr).
\end{aligned}
\eeq
Lemma\refer  {TaylorHtilde}, and Assertions\refeq  {equivonHtildeeq1} and\refeq  {equivonHtildeeq2}  imply that
\beno 
\frac 1 {2|\lam|} \wt \D_j (n,\lam) &\equivH {2M-1} & \frac {n_j+1} {2|\lam|}   
 \sum_{\ell=0} ^{2M} \frac {(2|\lam|)^\ell } {\ell!}   (\partial_{x_j}^\ell f)(|\lam|R(n),\lam\bigr)\\
 &&\qquad\qquad\qquad\qquad
 {}+\frac {n_j} {2|\lam|}  \sum_{\ell=0} ^{2M} \frac  { (-2|\lam|)^\ell } {\ell!}   (\partial_{x_j}^\ell f)(|\lam|R(n),\lam\bigr)\\
 &\equivH {2M-1} & \frac {2n_j+1}  {2|\lam|}    \sum_{\ell=0} ^{M} \frac {(2\lam)^{2\ell}} {(2\ell)!}    (\partial_{x_j}^{2\ell} f)(|\lam|R(n),\lam\bigr)\\
  &&\qquad\qquad\qquad\qquad
 {}+\frac {1} {2|\lam|} 
  \sum_{\ell=0} ^{M-1} \frac {(2|\lam|)^{2\ell+1} } {(2\ell+1)!}   (\partial_{x_j}^{2\ell+1} f)(|\lam|R(n),\lam\bigr).
\eeno
Let us define 
\beq
\label  {radialSdescribtiondemoeq1} 
\begin{aligned}
f_{2\ell} (x,\lam) &\eqdefa     \sum_{j=1}^d \frac {2^{2\ell-1}} {(2\ell)!} x_j\lam^{2\ell-2} \partial_{x_j}^{2\ell} f(x, \lam)\andf\\
f_{2\ell+1} (x,\lam) &\eqdefa  \sum_{j=1}^d \frac {2^{2\ell}} {(2\ell+1)!}  \lam^{2\ell} \partial^{2\ell+1}_{x_j} f(x, \lam).
\end{aligned}
\eeq
Clearly, all functions $f_\ell$ are supported in $[0,\infty[^d\times\{0\}\times\R$ and belong
to $\cS^d_+,$ and the above equality rewrites
\beq
\label  {radialSdescribtiondemoeq2}
\wh \Delta \Theta_f (n,\lam) \equivH {2M-1} - \sum_{\ell=1}^{2M} f_\ell \bigl(|\lam|R(n),\lam\bigr).
\eeq
Arguing by induction, it is  easy to establish that for any function~$f$ in~$\cS_d^+$ supported in~$[0,\infty[\times\{0\}\times \R$ and any integers $N$ and $p,$ 
the quantity $\|\wh\D^p \Theta_f\|_{N,0,\cS(\wh\H^d)}$ is finite.
Indeed, this is obvious for $p=0.$ Now, if the property holds true for some non negative integer~$p$ then, thanks to~\eqref{radialSdescribtiondemoeq2} and~\eqref{equivonHtildeeq3},
$$
\wh \Delta^{p+1} \Theta_f (n,\lam) \equivH {2M-1-2p}  -\sum_{\ell=1}^{2M} \wh \D^p \Theta_{f_\ell}(n,\lam).
$$
{}From\refeq{equivonHtildeeq0}, \refeq  {radialSdescribtiondemoeq2} and the induction hypothesis, it is clear that if we choose~$M$ greater than~$p$  then we get that 
$\|\wh\D^p \Theta_f\|_{N,0,\cS(\wh\H^d)}$ is finite for all integer $N.$
\medbreak
Let us next study the action of Operator~$\wh\cD_\lam$. From its definition in Lemma \ref {decaygivesregul1}, we gather~that
 $$\begin{aligned}
(\wh\cD_\lam \Theta_f)(\wh w)   &=  \frac {d} {d\lam}\bigl(  f \bigl(|\lam|R(n),\lam\bigr)\bigr) + \frac d {2\lam}  f \bigl(|\lam|R(n),\lam\bigr)  
+\frac 1{2 \lam}\sum_{j=1}^d \cD_j(n,\lam)\with \\
\cD_j(n,\lam)& \eqdefa 
n_j   f \bigl(|\lam| (R(n)-2\d_j),\lam\bigr) 
-(n_j+1) f \bigl(|\lam| (R(n)+2\d_j),\lam\bigr). 
\end{aligned}
$$
Lemma\refer  {TaylorHtilde}, and Assertions\refeq  {equivonHtildeeq1} and\refeq  {equivonHtildeeq2}    imply that 
 \ben
\nonumber\frac 1 {2\lam}  \cD _j(n,\lam) &  \equivH {2M-1} & -\frac {n_j+1} {2\lam} \sum_{\ell=0} ^{2M} \frac { (2|\lam|)^\ell } {\ell!} (\partial_{x_j}^\ell f)(|\lam|R(n),\lam\bigr)\\
\nonumber  &&\ \quad\qquad\qquad  {}+\frac {n_j} {2\lam} \sum_{\ell=0} ^{2M} \frac {(-2|\lam|)^\ell} {\ell!}    (\partial^\ell_{x_j} f )(|\lam|R(n),\lam\bigr)\\
\label  {radialSdescribtiondemoeq3}
  &\equivH {2M-1} & -  \sum_{\ell=0}^M \frac { (2\lam)^{2\ell-1 }} {(2\ell)!}  (\partial_{x_j}^{2\ell} f)(|\lam|R(n),\lam\bigr)\\
 \nonumber
  &&\ \quad\qquad\qquad  {}-\frac {2n_j+1} {2\lam} \sum_{\ell=0} ^{M-1} \frac {(2|\lam|)^{2\ell+1}} {(2\ell+1)!}    (\partial^{2\ell+1}_{x_j} f )(|\lam|R(n),\lam\bigr).
  \een
{}Applying the chain rule yields 
\beq
\label  {radialtypeSdescribtiondemoeq1}
\frac {d} {d\lam}\bigl(  f \bigl(|\lam|R(n),\lam\bigr)\bigr) 
=(\partial_\lam f)\bigl(|\lam|R(n),\lam\bigr)
+\sgn\lam \sum_{j=1} ^d (2n_j+1)(\partial_{x_j} f)\bigl(|\lam|R(n),\lam\bigr).
\eeq
Defining for $\ell\geq1$ the functions
\beno
\wt f_{2\ell}(x,\lam ) &
\eqdefa & \sum_{j=1}^d \frac {2^{2\ell-1}} {(2\ell)!}  \lam ^{2\ell-1} \partial_{x_j}^{2\ell} f(x, \lam )\andf \\
\wt f_{2\ell+1} (x,\lam ) &\eqdefa & \sum_{j=1}^d \frac {2^{2\ell}} {(2\ell+1)!}  x_j \lam ^{2\ell-1} \partial_{x_j}^{2\ell+1} f(x, \lam ),
\eeno
we get, using\refeq {radialSdescribtiondemoeq3} and\refeq  {radialtypeSdescribtiondemoeq1}, 
 $$
(\wh\cD_\lam \Theta_f)(\wh w) \equivH{2M-1} (\partial_\lam f)\bigl(|\lam| R(n),\lam) -\sum_{\ell=2} ^{2M} \wt f_\ell \bigl(|\lam| R(n),\lam\bigr) .
$$
{}From that relation, mimicking the induction proof for $\wh\D,$ we easily conclude that 
for any function $f$ in $\cS_d^+$ supported in~$[0,\infty[^d\times\{0\}\times \R,$ and any integer $p,$ 
the quantity $\|\wh\cD_\lam^p \Theta_f\|_{N,0,\cS(\wh\H^d)}$ is finite for all integer $N.$
This  completes the proof Theorem\refer{radialtypeSdescribtion} in that particular case.
\medbreak
Next, let us investigate the case when the function~$f$ of $\cS_d^+$ is supported in~$[r_0,\infty[^d\times \ZZ^d\times \R$ for some positive~$r_0$ and satisfies \eqref{eq:fsym}. Then, by definition of the operator~$\wh \D$, we have for all $\wh w=(n,m,\lam)$ in~$\wt\H^d,$ denoting $k\eqdefa m-n$,
\beno
-\wh \D \Theta_f (\wh w)  &\eqdefa &   \frac {1} {2 |\lam|} \Bigl( (|n+m| +d)  f\bigl (|\lam|R(n,m),k,\lam\bigr)
-\sum_{j=1}^d  \wt \D_j (\wh w)\Bigr)\with\\\
 \wt \D_j (\wh w) & \eqdefa &
\sqrt{(n_j+1)(m_j+1)} \, f \bigl(|\lam|(R(n,m)+2\d_j),k, \lam\bigr) \\
&&\qquad\qquad\qquad\qquad\qquad\qquad{}
 + \sqrt {n_jm_j} \, f \bigl(|\lam|(R(n,m)-2\d_j),k,\lam\bigr).
\eeno
Compared to\refeq {radialSdescribtiondemoeq0}, the computations get wilder, 
owing to the square roots in the above formula. 
Let $M$ be an integer (to be suitably chosen later on). Lemma\refer  {TaylorHtilde}, and Assertions\refeq  {equivonHtildeeq1} and\refeq  {equivonHtildeeq2}  imply that 
$$\longformule{
\frac 1 {2|\lam|} \wt \D_j (\wh w) \equivH {2M-1}  \frac {\sqrt{(n_j+1)(m_j+1)}} {2|\lam|}   
 \sum_{\ell=0} ^{2M} \frac {(2|\lam|)^\ell } {\ell!}   (\partial_{x_j}^\ell f)\bigl(|\lam|R(n,m),k,\lam\bigr)
 }
 {
 {}+\frac {\sqrt{n_jm_j}} {2|\lam|}  \sum_{\ell=0} ^{2M} \frac  { (-2|\lam|)^\ell } {\ell!}   
 (\partial_{x_j}^\ell f)\bigl(|\lam|R(n,m),k,\lam\bigr).
 }
 $$
 Defining
 \beq
 \label {typeSgenedemoeq1000}
 \al^\pm(p,q) \eqdefa \sqrt{(p+1)(q+1)}\pm\sqrt {pq},
\eeq
for nonnegative integers~$p$ 	and~$q$, we get
 \beq
 \label {typeSgenedemoeq10}
\frac 1 {2|\lam|} \wt \D_j (\wh w) \equivH {2M-1} \wt \D_j^{0} (\wh w) +\wt \D_j^{1} (\wh w)
\eeq
$$
\begin{aligned}
\with\wt \D_j^{0} (\wh w) & \eqdefa  \frac {\al^+(n_j,m_j) }  {2|\lam|}    \sum_{\ell=0} ^{M} \frac {(2\lam)^{2\ell}} {(2\ell)!}    (\partial_{x_j}^{2\ell} f)\bigl(|\lam|R(n,m),k,\lam\bigr)\\
\andf\wt \D_j^{1} (\wh w) & \eqdefa  \frac {\al^-(n_j,m_j) }  {2|\lam|}  
  \sum_{\ell=0} ^{M-1} \frac {(2|\lam|)^{2\ell+1} } {(2\ell+1)!}   (\partial_{x_j}^{2\ell+1} f)\bigl(|\lam|R(n,m),k,\lam\bigr).
 \end{aligned}
 $$
 Now let us compute an expansion of~$\al_j^\pm(n,m)$ with respect to~$n_j+m_j+1$ and~$n_j-m_j$.  
Let~$p$ and~$q$ be two integers and let us write
\beno
(p+1)(q+1) & = & pq+ p+q+1\andf\\
pq & =  & \frac 1 4 \bigl((p+q+1)^2 -2(p+q+1) + 1-(p-q)^2 \bigr)\,.
\eeno
We get that
\beno
\sqrt {(p+1)(q+1)} &= & \frac 12 (p+q+1) \sqrt {1+\frac 2 {p+q+1} +\frac {1-(p-q)^2} {(p+q+1)^2}}\\
\andf\sqrt {pq} &= & \frac 12 (p+q+1) \sqrt {1-\frac 2 {p+q+1} +\frac {1-(p-q)^2} {(p+q+1)^2}}\,\cdotp
\eeno
Let us introduce the notation~$f(p,q)=\cO_M(p,q)$ to mean that  for some constant $C,$ there holds 
$$
|f(p,q)|\leq C\biggl(\frac{1}{(p+q+1)^M}+\frac{|p-q|^{2M+2}}{(p+q+1)^{2M+1}}\biggr) \cdotp
$$
Using the following Taylor expansion with $K=2M$:
$$
\sqrt{1+u}=1+\sum_{\ell_1=1}^{K} a_{\ell_1} u^{\ell_1} + (K+1)a_{K+1}u^{K+1}\int_0^1(1+tu)^{-K-\frac12}(1-t)^{K}\,dt,
$$
 we gather that 
$$
\begin{aligned}
\sqrt {(p+1)(q+1)} &=
 \frac 12 (p+q+1) \biggl(1+\sum_{\ell_1=1} ^{2M}
a_{\ell_1}\Bigl( \frac 2 {p+q+1} +\frac {1-(p-q)^2} {(p+q+1)^2}\Bigr) ^{\ell_1}\biggr)+\cO_{2M}(p,q)\\\andf 
\sqrt {pq} &= 
\frac 12 (p+q+1) \biggl(1+\sum_{\ell_1=1} ^{2M}
a_{\ell_1}\Bigl(- \frac 2 {p+q+1} +\frac {1-(p-q)^2} {(p+q+1)^2}\Bigr) ^{\ell_1}\biggr)+\cO_{2M}(p,q).
\end{aligned}$$
Now we can compute the expansion of~$\al^\pm(p,q)$. Newton's formula gives
\beq
 \label {typeSgenedemoeq12}
 \begin{aligned}
\al^+(p,q)  &=  p+q+1 +\sumetage {1\leq \ell_1\leq 2M} {2\ell_2\leq \ell_1} a_{\ell_1} \begin{pmatrix}  \ell_1 \\ 2\ell_2\end{pmatrix} \frac {4^{\ell_2} \bigl(1-(p-q)^2\bigr)^{\ell_1-2\ell_2} } {( p+q+1)^{2\ell_1-2\ell_2-1}}+\cO_{2M}(p,q) \\
\al^-(p,q) &= 2\sumetage {1\leq \ell_1\leq 2M} {2\ell_2+1\leq \ell_1} a_{\ell_1} \begin{pmatrix}  \ell_1 \\ 2\ell_2+1\end{pmatrix} \frac {4^{\ell_2} \bigl(1-(p-q)^2\bigr)^{\ell_1-2\ell_2-1} } {( p+q+1)^{2\ell_1-2\ell_2-2}}+\cO_{2M}(p,q).
\end{aligned}
\eeq

In the above expansion, some terms  that turn out to be  $\cO_{2M}(p,q)$ are kept 
for notational simplicity.  Now, one may check that  for all functions~$\theta$ and 
$\theta'$ supported in~$[r_0,\infty[^d\times \ZZ^d\times \R$ and any integers $M_1$ and $M_2,$ we have
for all $j\in\{1,\cdots,d\},$
\beq
\label {typeSgenedemoeq13}
\bigl( f=\cO_{M_1} \!\andf \! \theta  \equivH {M_2} \theta'\bigr)  \Longrightarrow f(n_j,m_j) \theta(\wh w) \equivH {M_1+M_2}  f(n_j,m_j) \theta'(\wh w).
\eeq
Then  Assertion\refeq {typeSgenedemoeq12} implies that  for any function~$g $ in~$\cS_d^+$ supported in~$[r_0,\infty[^d\times \ZZ^d\times \R$, and any~$j$ in~$\{1,\cdots,d\}$, we have
\beq
\label{eq:alpha+}
\begin{aligned}
&\al^+(n_j,m_j) \Theta_g(\wh w)  \equivH {2M-1} \biggl( n_j\!+\!m_j\!+\!1
\\
&\qquad\qquad\qquad {}+\sumetage {1\leq \ell_1\leq 2M} {2\ell_2\leq \ell_1} \! a_{\ell_1} \begin{pmatrix}  \ell_1 \\ 2\ell_2\end{pmatrix} \frac {4^{\ell_2} \bigl(1-(n_j-m_j)^2\bigr)^{\ell_1-2\ell_2} } {( n_j+m_j+1)^{2\ell_1-2\ell_2-1}}\biggr) \Theta_g(\wh w)\andf
\end{aligned}
\vspace*{-3mm}
\eeq
\begin{equation}\label{eq:alpha-}
\al^-(n_j,m_j)  \theta_g(\wh w)   \equivH {2M-1} 2\biggl(\sumetage {1\leq \ell_1\leq 2M} {2\ell_2+1\leq \ell_1} \!a_{\ell_1}\! \begin{pmatrix}  \ell_1 \\ 2\ell_2\!+\!1\end{pmatrix} \frac {4^{\ell_2}
 \bigl(1-(n_j-m_j)^2\bigr)^{\ell_1-2\ell_2-1} } {(n_j\!+\!m_j\!+\!1)^{2\ell_1-2\ell_2-2}}\biggr) \Theta_g(\wh w).
\end{equation}
 Using\refeq{typeSgenedemoeq10}, this gives
$$\begin{aligned}
\wt \D_j^{(0)} (\wh w) &\equivH{2M-1}  \biggl(\frac {n_j+m_j+1} {2|\lam|}\biggr) \Theta_f(\wh w) +\sum_{\ell=0}^M \Theta_{f_{j,2\ell}}(\wh w) \with \\
f_{j,0} (x,k,\lam) &\eqdefa  \sumetage {1\leq \ell_1\leq 2M} {2\ell_2\leq \ell_1}\!\! a_{\ell_1}\!\begin{pmatrix}  \ell_1 \\ 2\ell_2\end{pmatrix} 2^{2\ell_2-1} \bigl(1-k_j^2\bigr)^{\ell_1-2\ell_2} \frac { \lam^{2(\ell_1-\ell_2-1)}} {x_j^{2\ell_1-2\ell_2-1}} f(x,k,\lam) 
\end{aligned}
$$
 and, if  $1\leq\ell\leq M,$
 $$
f_{j,2\ell} (x,k,\lam) \eqdefa  \frac1{(2\ell)!}\sumetage {0\leq \ell_1\leq 2M} {2\ell_2\leq \ell_1} a_{\ell_1}\begin{pmatrix}  \ell_1 \\ 2\ell_2\end{pmatrix} 2^{2\ell_2-1+2\ell} \bigl(1-k_j^2\bigr)^{\ell_1-2\ell_2} \frac { \lam^{2(\ell+\ell_1-\ell_2-1)}} {x_j^{2\ell_1-2\ell_2-1}} \partial_{x_j}^{2\ell}f(x,k,\lam). $$
Similarly,
\beno
\wt \D_j^{(1)} (\wh w) &\equivH{2M-1} & \sum_{\ell=0} ^{M-1} \Theta_{f_{j,2\ell+1}} \with \\
f_{j,2\ell+1} (x,k,\lam) &\eqdefa &  \frac1{(2\ell+1)!}\sumetage {1\leq \ell_1\leq 2M} {2\ell_2+1\leq \ell_1} a_{\ell_1}\begin{pmatrix}  \ell_1 \\ 2\ell_2+1\end{pmatrix} 4^{\ell_2+\ell} \bigl(1-k_j^2\bigr)^{\ell_1-2\ell_2-1}\\
&&\qquad\qquad\qquad\qquad\qquad\qquad
{}\times \frac { \lam^{2(\ell+\ell_1-\ell_2-1)}} {x_j^{2\ell_1-2\ell_2-2}} \partial_{x_j}^{2\ell+1}f(x,k,\lam) .
\eeno
{}From the definition of Operator~$\wh \D$, we thus infer that there exist functions $f_\ell$
of $\cS_d^+$ supported in $[r_0,\infty[^d\times\Z^d\times\R$  and satisfying \eqref{eq:fsym}, such that for all $M\geq0,$ we have
$$
\wh\D \Theta_f \equivH {2M-1}  \sum_{\ell=0}^{2M} \Theta _{f_\ell}.
$$
At this stage,  one may prove  by induction,  as in the previous case, that 
$\|\wh\D^p \Theta_f\|_{N,0,\cS(\wh\H^d)}$ is finite for all integers $N$ and $p.$
\medbreak
Let us finally study the action  of~$\wh\cD_\lam$.  
  {}From its definition,  setting $k=m-n,$ we get
 $$
 \displaylines{
(\wh\cD_\lam  \Theta_f)(\wh w)   =  \frac {d} {d\lam}\bigl(  f \bigl(|\lam|R(n,m),k,\lam\bigr)\bigr) + \frac d {2\lam}  f \bigl(|\lam|R(n,m),k,\lam\bigr)  
+\frac 1{2 \lam}\sum_{j=1}^d \cD_j(\wh w)\with
\cr
\cD_j(\wh w)\eqdefa 
\sqrt {n_jm_j}   f \bigl(|\lam| (R(n,m)-2\d_j),k,\lam\bigr) 
-\sqrt {(n_j+1)(m_j+1)}  f \bigl(|\lam| (R(n,m)+2\d_j),k,\lam\bigr). 
}
$$
 The chain rule implies that 
\beq
\label  {typeSgenedemoeq1}
\begin{aligned}
&\frac {d} {d\lam}\bigl(  f \bigl(|\lam|R(n,m),k,\lam\bigr)\bigr) 
=(\partial_\lam f)\bigl(|\lam|R(n,m),k,\lam\bigr)\\
&\qquad\qquad\qquad\qquad\qquad\qquad
{}+\sgn\lam \sum_{j=1} ^d (n_j+m_j+1)(\partial_{x_j} f)\bigl(|\lam|R(n,m),k, \lam\bigr).
\end {aligned}
\eeq
Combining Lemma\refer  {TaylorHtilde}, and Assertions\refeq  {equivonHtildeeq1} and\refeq  {equivonHtildeeq2}   yields 
 $$\displaylines{\quad
-\frac 1 {2\lam}  \cD _j(\wh w)  \equivH {2M-1}  \alpha^-(n_j,m_j)
\sum_{\ell=0}^M \frac { (2\lam)^{2\ell-1 }} {(2\ell)!}  (\partial_{x_j}^{2\ell} f)(|\lam|R(n,m),k,\lam\bigr)\hfill\cr\hfill
  +{\alpha^+(n_j,m_j)}\sgn\lam\sum_{\ell=0} ^{M-1}\frac {(2\lam)^{2\ell}}{(2\ell+1)!}    (\partial^{2\ell+1}_{x_j} f )(|\lam|R(n,m),k,\lam\bigr).\quad}
  $$
  Therefore, we have 
$$
\displaylines{(\wh\cD_\lam  \Theta_f)(\wh w)   \equivH {2M-1}   (\partial_\lam f)\bigl(|\lam|R(n,m),k,\lam\bigr)
+\frac{1}{2\lam} \biggl(d-\sum_{j=1}^d\alpha^-(n_j,m_j)\biggr) f \bigl(|\lam|R(n,m),k,\lam\bigr)\hfill\cr\hfill
+\sgn\lam \sum_{j=1} ^d \bigl(n_j+m_j+1 - \alpha^+(n_j,m_j)\bigr)(\partial_{x_j} f)\bigl(|\lam|R(n,m),k, \lam\bigr)
\hfill\cr\hfill- \alpha^-(n_j,m_j)
\sum_{\ell=1}^M \frac { (2\lam)^{2\ell-1 }} {(2\ell)!}  (\partial_{x_j}^{2\ell} f)(|\lam|R(n,m),k,\lam\bigr)\hfill\cr\hfill
  -{\alpha^+(n_j,m_j)}\sgn\lam\sum_{\ell=1} ^{M-1}\frac {(2\lam)^{2\ell}}{(2\ell+1)!}    (\partial^{2\ell+1}_{x_j} f )(|\lam|R(n,m),k,\lam\bigr).}
$$
Hence, using \eqref{eq:alpha+} and \eqref{eq:alpha-} and noticing that the coefficient $\ds \!a_{\ell_1}$ involved in the expansion of $ \al^\pm (n_j,m_j)$ is equal to $1/2$, we conclude that
there exist some functions $\wt f_j$, $f^\flat _j$ and  $f^\sharp _{j,\ell}$ of~$\cS_d^+,$ supported in $[r_0,\infty[^d\times\Z^d\times\R$ and satisfying \eqref{eq:fsym}
so that 
$$
\displaylines{(\wh\cD_\lam  \Theta_f)(\wh w)   \equivH {2M-1}  (\Theta_{\partial_\lam f} )(\wh w) +\sum_{j=1}^d  (\Theta_{\wt f_j} )(\wh w) - \sum_{j=1}^d (\Theta_{f^\flat _j} )(\wh w) - \sum_{\ell=1} ^{2M} \sum_{j=1}^d (\Theta_{f^\sharp _{j,\ell}} )(\wh w) 
\,,}
$$
where 
$$ \wt f_j (x,k,\lam) \eqdefa \sumetage {2\leq \ell_1\leq 2M} {2\ell_2+1\leq \ell_1} \!a_{\ell_1}\! \begin{pmatrix}  \ell_1 \\ 2\ell_2+1\end{pmatrix} \frac {4^{\ell_2}
 \bigl(1-k_j^2\bigr)^{\ell_1-2\ell_2-1}  \lam^{2\ell_1-2\ell_2-3 }} {x_j^{2\ell_1-2\ell_2-2}} f (x,k,\lam)\, ,$$
$$ f^\flat _j(x,k,\lam) \eqdefa  \sumetage {1\leq \ell_1\leq 2M} {2\ell_2\leq \ell_1} \! a_{\ell_1} \begin{pmatrix}  \ell_1 \\ 2\ell_2\end{pmatrix} \frac {4^{\ell_2} \bigl(1-k_j^2\bigr)^{\ell_1-2\ell_2} \lam^{2\ell_1-2\ell_2-1}} {x_j ^{2\ell_1-2\ell_2-1}}\, \partial_{x_j}f (x,k,\lam)\, ,$$
$$ f^\sharp _{j,2\ell}(x,k,\lam) \eqdefa  \biggl(\sumetage {1\leq \ell_1\leq 2M} {2\ell_2+1\leq \ell_1} \!a_{\ell_1}\! \begin{pmatrix}  \ell_1 \\ 2\ell_2\!+\!1\end{pmatrix} \frac {2^{2\ell_2+2\ell}
 \bigl(1-k_j^2\bigr)^{\ell_1-2\ell_2-1} } {x_j^{2\ell_1-2\ell_2-2}}\biggr)\frac { \lam^{2\ell + 2\ell_1-2\ell_2-3 }} {(2\ell)!}  (\partial_{x_j}^{2\ell} f)(x,k,\lam)\, $$
and 
$$
\longformule{
 f^\sharp _{j,2\ell+1}(x,k,\lam) \eqdefa  2 \biggl( x_j
\\+\sumetage {1\leq \ell_1\leq 2M} {2\ell_2\leq \ell_1} \! a_{\ell_1} \begin{pmatrix}  \ell_1 \\ 2\ell_2\end{pmatrix} \frac {4^{\ell_2} \bigl(1-k_j^2\bigr)^{\ell_1-2\ell_2} \lam^{2\ell_1-2\ell_2 }} {x_j^{2\ell_1-2\ell_2-1}}\biggr)
}
{�{}\times \frac {(2\lam)^{2\ell -1}}{(2\ell+1)!}    (\partial^{2\ell+1}_{x_j} f )(x,k,\lam)\, .
}
$$
At this stage, one  can complete  the proof as in the previous cases.\qed


\medbreak

It will be useful to give the following asymptotic description of the operators $\wh \D$ and $\wh \cD_\lam$ when $\lam$ tends to 0:

\begin{proposition}
\label {whDeltaover0}{\sl
For any  function $f$ in   $\cS_1^+$ supported in~$[r_0,\infty[\times \ZZ\times\R$  for some positive~$r_0,$
the extension to $\wh \D \Theta_f$ and $\wh \cD_\lam \Theta_f$ to $\wh\H^d_0$ is given by
$$\begin{aligned}
(\wh \D \Theta_f )(\dot x, k)&= \dot x\partial^2_{\dot x\dot x}f(\dot x,k,0)+\partial_{\dot x}f(\dot x,k,0)
-\frac{k^2}{4\dot x}f(\dot x,k,0)\andf\\
(\wh \cD_\lam \Theta_f )(\dot x, k)&= \partial_\lam f(\dot x,k,0).\end{aligned}
$$}
\end{proposition}

\begin{proof}
For expository purpose, we omit the dependency on $k,$ for $f.$ Then we have by definition of  $\Theta_f$ and $\wh\D,$
for all $(n,n+k,\lam)$ in~$\wt\H^d$ with positive~$\lam$,
$$\displaylines{\quad
-2\lam^2\wh\D\Theta_f(n,n+k,\lam)=\lam(2n+k+1)f(\lam(2n+k+1),\lam)\hfill\cr\hfill
-\lam\sqrt{(n+1)(n+k+1)}f(\lam(2n+k+3),\lam)-\lam\sqrt{n(n+k)}
f(\lam(2n+k-1),\lam).\quad}$$
Denoting~$\dot x= 2\lam n,$ the above equality rewrites
\beq
\label  {whDeltaover0demoeq1}
\begin {aligned}
-2\lam^2\wh \D \Theta_f (\wh w)  & =    \wt \D^1(\wh w)-\wt \D^2(\wh w)-\wt \D^3(\wh w) \with \\
\wt \D^1 (\wh w)& \eqdefa  \bigl(\dot x+\lam(k+1)\bigr)  f\bigl (\dot x+\lam (k+1)\bigr)\,,
\\
\wt \D^2 (\wh w)& \eqdefa
\sqrt{\Bigl( \frac {\dot x} 2 + \lam\Bigr)\Bigl(\frac {\dot x} 2 + \lam(k+1)\Bigr)} 
\:f\bigl (\dot x+\lam (k+3)\bigr),\\
\wt \D^3 (\wh w)& \eqdefa
\sqrt{\frac {\dot x} 2\Bigl( \frac {\dot x} 2 + \lam k\Bigr)} \:f\bigl (\dot x+\lam (k-1)\bigr).
\end {aligned}
\eeq
In what follows, we shall use repeatedly the following asymptotic expansion for $y>0$ and $\eta$ in $]-y,y[$:
\begin{equation}\label{eq:y}
\sqrt{y+\eta}=\sqrt y+\frac\eta{2\sqrt y}-\frac{\eta^2}{8y\sqrt y}+\sqrt y\,\cO\biggl(\!\Bigl(\frac\eta y\Bigr)^{\!\!3}\!\biggr)\cdotp
\end{equation}
Let us compute the second order expansions of $\wt \D^1(\wh w),$ $\wt \D^2(\wh w)$ and $\wt \D^3(\wh w)$
 with respect to~$\lam,$ for fixed (and positive) value of $\lam n.$  We have
\begin{multline}
\label  {whDeltaover0demoeq2}
\wt \D^1 (\wh w)=  \dot x f(\dot x,0) + \Bigl((k+1)\bigl(f(\dot x,0)+\dot x\partial_{\dot x}f(\dot x,0)\bigr) +\dot x\partial_{\lam}f(\dot x,0)\Bigr)
\lam \\
{}+\biggl(\frac12\dot x\partial^2_{\lam\lam}f(\dot x,0) +\frac{(k+1)^2}2\Bigl(\dot x\partial^2_{\dot x\dot x}f(\dot x,0)+2\partial_{\dot x}f(\dot x,0)\Bigr) \\
{}+ (k+1) \Bigl(\partial_{\lam}f(\dot x,0)+ \dot x\partial^2_{\dot x\lam}f(\dot x,0)\Bigr)\biggr)\lambda^2
+\cO(\lambda^3)\cdotp
\end{multline}
In order to find out  the second order expansions of $\wt\D^2(\wh w)$ and $\wt\D^3(\wh w),$
we shall use the fact that, denoting $\dot y=\dot x/2$ and using \eqref{eq:y}, 
$$\displaylines{
\wt\D^2(\wh w)=\biggl(\sqrt{\dot y}+\frac{\lambda}{2\sqrt{\dot y}}-\frac{\lambda^2}{8\dot y\sqrt{\dot y}}\biggr)
\biggl(\sqrt{\dot y}+\frac{(k+1)\lambda}{2\sqrt{\dot y}}-\frac{(k+1)^2\lambda^2}{8\dot y\sqrt{\dot y}}\biggr)
\hfill\cr\hfill\times\biggl(f(\dot x,0)+\bigl(\partial_\lam f(\dot x,0)+(k+3)\partial_{\dot x}f(\dot x,0)\bigr)\lam
\hfill\cr\hfill+\Bigl(\frac12\partial_{\lam\lam}^2f(\dot x,0)+(k+3)\partial^2_{\dot x\lam}f(\dot x,0)+\frac{(k\!+\!3)^2}2\partial^2_{\dot x\dot x}f(\dot x,0)\!\Bigr)\lam^2\biggr)+\cO(\lam^3)\cdotp}
$$
Hence, we get at the end, replacing $\dot y$ by its value, 
\begin{multline}\label{eq:Dp2}
\wt\D^2(\wh w)=\frac{\dot x}2f(\dot x,0)+
\biggl(\Bigr(1+\frac k2\Bigr)f(\dot x,0)+\Bigl(\frac{k+3}2\Bigr)\dot x\partial_{\dot x}f(\dot x,0)
+\frac12{\dot x}\partial_\lam f(\dot x,0)\biggr)\lambda
\\+\biggl(\frac{(k\!+\!3)^2}4\dot x\partial^2_{\dot x\dot x}f(\dot x,0)+\biggl(\frac{k+3}2\biggr)
\dot x\partial^2_{\dot x\lam}f(\dot x,0)
+\frac{\dot x}4\partial^2_{\lam\lam}f(\dot x,0)\\
+\Bigl(1+\frac k2\Bigr)\bigl((k+3)\partial_{\dot x}f(\dot x,0)+\partial_\lam f(\dot x,0)\bigr)-\frac{k^2}{4\dot x} f(\dot x,0)\biggr)\lam^2+\cO(\lam^3).\end{multline}
Similarly, we have
$$\displaylines{
\wt\D^3 (\wh w)=\sqrt{\dot y}\biggl(\sqrt{\dot y}+\frac{k\lambda}{2\sqrt{\dot y}}-\frac{k^2\lambda^2}{8\dot y\sqrt{\dot y}}\biggr)
\biggl(f(\dot x,0)+\bigl(\partial_\lam f(\dot x,0)+(k-1)\partial_{\dot x}f(\dot x,0)\bigr)\lam
\hfill\cr\hfill+\Bigl(\frac12\partial_{\lam\lam}^2f(\dot x,0)+(k-1)\partial^2_{\dot x\lam}f(\dot x,0)+\frac{(k-1)^2}2\partial^2_{\dot x\dot x}f(\dot x,0)\Bigr)\lam^2\biggr)+\cO(\lam^3),}
$$
whence, 
\begin{multline}\label{eq:Dp3}
\wt\D^3(\wh w)=\frac{\dot x}2f(\dot x,0)+
\biggl(\frac k2 f(\dot x,0)+\Bigl(\frac{k-1}2\Bigr)\dot x\partial_{\dot x}f(\dot x,0)
+\frac12{\dot x}\partial_\lam f(\dot x,0)\biggr)\lambda
\\+\biggl(\frac{(k-1)^2}4\dot x\partial^2_{\dot x\dot x}f+\biggl(\frac{k-1}2\biggr)\dot x\partial^2_{\dot x\lam}f(\dot x,0)
+\frac{\dot x}4\partial^2_{\lam\lam}f(\dot x,0)
\\+\frac k2\bigl((k-1)\partial_{\dot x}f(\dot x,0)+\partial_\lam f(\dot x,0)\bigr)-\frac{k^2}{4\dot x} f(\dot x,0)\biggr)\lam^2+\cO(\lam^3).
\end{multline}
Inserting the above relations in\refeq{whDeltaover0demoeq1}, we discover that 
the zeroth and first order terms in the expansion cancel, and that 
$$
-2\lam^2\wh \D \Theta_f (\wh w) = \biggl(\frac{k^2}{2\dot x}f(\dot x,0)-2\partial_{\dot x}f(\dot x,0)-2\dot x\partial^2_{\dot x\dot x}f(\dot x,0)\biggr)\lambda^2+\cO(\lambda^3),
$$
which ensures that 
$$
\lim_{\lam\to0}\wh\D\Theta_f(\wh w)= \dot x\partial^2_{\dot x\dot x}f(\dot x,0)+\partial_{\dot x}f(\dot x,0)-\frac{k^2}{4\dot x}f(\dot x,0).
 $$
 The proof for Operator $\wh\cD_\lam$ is quite similar: from the definition of $\wh\cD_\lam$ and the chain rule, we discover that
 for all $(n,n+k,\lam)$ in $\wt\H^d$ with $\lam>0,$
$$\displaylines{\quad
\wh\cD_\lam\Theta_f(\wh w)=(2n+k+1)\partial_{\dot x}f(\lam(2n+k+1),\lam) +\partial_\lam f(\lam(2n+k+1),\lam)
\hfill\cr\hfill+\frac1{2\lam}\bigl(f(\lam(2n+k+1),\lam)+\sqrt{n(n\!+\!k)}f(\lam(2n+k-1),\lam)
-\sqrt{(n+1)(n\!+\!k\!+\!1)}f(\lam(2n+k+3),\lam)\bigr).}
$$
Therefore,  assuming that $\dot x\eqdefa2\lam n>0,$  we get
\begin{multline}
\wh\cD_\lam\Theta_f(\wh w)= \partial_\lam f(\dot x+(k+1)\lam,\lam)+
\frac1{\lam}(\dot x+(k+1)\lam)\partial_{\dot x}f(\dot x+(k+1)\lam,\lam) \\
+\frac1{2\lam^2}\Bigl(\lam f(\dot x+(k+1)\lam,\lam) +\wt\D^3(\wh w)-\wt\D^2(\wh w)\Bigr)\cdotp
\end{multline}
Because
$$
 f(\dot x+(k+1)\lam,\lam) = f(\dot x,0)+\bigl((k+1)\partial_{\dot x}f(\dot x,0)+\partial_\lam f(\dot x,0)\bigr)\lam+\cO(\lam^2)
 $$
and 
$$\displaylines{\quad
(\dot x+(k+1)\lam)\partial_{\dot x}f(\dot x+(k+1)\lam,\lam) 
=\dot x\partial_{\dot x}f(\dot x,0)\hfill\cr\hfill
+\bigl((k+1)(\partial_{\dot x}f(\dot x,0)+\dot x\partial^2_{\dot x\dot x} f(\dot x,0))+\dot x\partial^2_{\dot x\lam} f(\dot x,0)\bigr)\lam+\cO(\lambda^2),}
$$
we get at the end, taking advantage of \eqref{eq:Dp2} and \eqref{eq:Dp3},
 $$\wh\cD_\lam\Theta_f(\wh w)= \partial_\lam f(\dot x,0)+\cO(\lambda),$$
which completes the proof.
\end{proof}

\medbreak

\section {Examples of tempered distributions} 
\label {examplescS'}

A  first class of examples will be given by  the functions 
belonging to the  space~$L^1_M(\wh\H^d)$ of Definition\refer {definL1Moderated}. 
This is exactly what states Theorem\refer{indentifonctdistritemp}
that we are going to prove now.
Inequality \eqref{eq:boundf} just follows from the definition of the semi-norms on $\wh \H^d.$
So let us focus on the proof of the first part of the statement. 
Let~$f$ be a function of~$ L^1_M(\wh \H^d) $  such that~$\iota(f)=0$. We claim that $f=0$ a.e. 
Clearly, it is enough to prove  for all $K>0$ and $b>a>0,$ we have 
\[\int_{\wh\cC_{a,b,K}}|f(\wh w)|\,d\wh w=0\]
where  $\wh\cC_{a,b,K}\eqdefa\{(n,m,\lambda)\in\wh\H^d\,:\,  
|\lambda|(|n+m|+d)\leq K,\; |n-m|\leq K \ \hbox{and}\ a\leq|\lambda|\leq b\bigr\}\cdotp$
\medbreak
To this end, we introduce the bounded function~: 
$$g\eqdefa \frac {\overline f}{|f|}\,{\bf 1}_{f\not=0}\,{\bf 1}_{\wh\cC_{a,b,K}}$$
and smooth it out with respect to $\lambda$  by setting 
$$
g_{\e} \eqdefa \chi_{\e}\star_\lambda g
$$
where $\chi_\e\eqdefa\ep^{-1}\chi(\ep^{-1}\cdot)$ and
$\chi$ stands for some smooth even function on $\R,$ supported 
in the interval $[-1,1]$  and with integral $1.$ 
\medbreak
Note that by definition, $g$ is supported in the set $\wh\cC_{a,b,K}.$ 
Therefore,  if $\ep<a$ then $g_\e$ is supported in $\wh\cC_{a-\ep,b+\ep,K}.$ 
This readily ensures that $\|g_\e\|_{N,0,\cS(\wh\H^d)}$ is finite for all integer $N$
(as regards  the action of operator $\wh\Sigma_0,$ note that $g_\ep(n,m,\lambda)=0$ 
whenever $|\lambda|<a-\ep$). 

In order to prove that $g_\ep$ belongs to $\cS(\wh\H^d),$ it suffices to use the following 
lemma the proof of which is left to the reader:
\begin{lemma} {\sl Let $h$ be a smooth function on $\wh\H^d$ with
support  in $\{(n,m,\lambda)\,:\, |\lambda|\geq a\}$ for some $a>0.$
If $h$ and all derivatives with respect to $\lambda$  have fast decay, 
that is have finite semi-norm $\|\cdot\|_{N,0,\cS(\wh\H^d)}$ for all integer $N,$ 
then the same properties hold true for $\wh\cD_\lambda h$ and~$\wh\Delta h.$}
\end{lemma}
Because $g_\ep$ is in $\cS(\wh\H^d)$ for all $0<a<\ep,$  our assumption 
on $f$ ensures that we have
 \[
 I_{\e}\eqdefa \int_{\wh\H^d} f(\wh w)\,g_{\e}(\wh w)\,d\wh w=0.
 \]
 Now, we notice that whenever $0<\e\leq a/2,$ we have for all $(n,m,\lambda)\in\wt\H^d$ and $\lambda'\in\R,$
 \[
 \frac1\e\chi\biggl(\frac {\lambda-\lambda'}\e\biggr) g(n,m,\lambda')\,f(n,m,\lambda)
 = \frac1\e\chi\biggl(\frac {\lambda-\lambda'}\e\biggr)    g(n,m,\lambda') 
 ({\bf 1}_{\wh\cC_{a/2,b+a/2,K}}f)(n,m,\lambda),\] 
 which guarantees that
 \[
 \int_{\wh\H^{d}\times \R}
 \chi_\ep(\lambda-\lambda') |g(n,m,\lambda')|\,|f(n,m,\lambda)|\,d\wh w\,d\lambda'
 \leq \|\chi\|_{L^1} \| {\bf 1}_{\wh\cC_{a/2,b+a/2,K}}f\|_{L^1}<\infty.
 \]
 Therefore applying  Fubini theorem, remembering that~$\chi$ is an even function and
 exchanging the notation $\lambda$ and $\lambda'$ in the second line below,
 \begin{eqnarray*}
 I_{\e} & =  & \int_{\wh\H^d} f(n,m,\lambda)\biggl(\int_\R\chi_{\e}(\lambda-\lambda')
 g(n,m,\lambda')\,d\lambda'\biggr)\,d\wh w\\
  & =  & \int_{\wh\H^d} g(n,m,\lambda)
  \biggl(\int_\R \chi_\e(\lambda-\lambda') (1_{\wh\cC_{a/2,b+a/2,K}}f)(n,m,\lambda')\,d\lambda'\biggr)
  d\wh w\\ & = & \int_{\wh\H^d} \bigl(\chi_{\e}\star ({\bf 1}_{\wh\cC_{a/2,b+a/2,K}}f)\bigr)(\wh w)   g(\wh w)\,d\wh w.
 \end{eqnarray*}
 The standard density theorem for convolution in $\R$ ensures that for all $(n,m)$ in~$\N^{2d},$ we have
 $$
 \lim_{\e\rightarrow 0}\int_\R \bigl|\chi_\ep\star  ({\bf 1}_{\wh\cC_{a/2,b+a/2,K}}f)(n,m,\lambda) -({\bf 1}_{\wh\cC_{a/2,b+a/2,K}}f)(n,m,\lambda)\bigr|\,d\lambda=0.
 $$
Hence, because the supremum of $g$ is bounded by $1,$ we get
 \begin{eqnarray*}
 0 & = & \lim_{\e\rightarrow 0} I_{\e}\\
  &= & \int_{\wh\H^d}   {\bf 1}_{\wh\cC_{a/2,b+a/2,K}}f(\wh w)\, g(\wh w)\,d\wh w\\
  & = & \int_{\wh\cC_{a,b,K}}|f(\wh w)|\,d\wh w,
 \end{eqnarray*}
 which completes the proof of Theorem\refer{indentifonctdistritemp}.\qed

\medbreak

Let us  prove  Proposition\refer {dimhomoconcretwhH} which claims that the  functions 
$$
f_\gamma (n,m,\lam) \eqdefa \bigl(|\lam| (2|m|+d)\bigr) ^{-\g} \, \delta_{n,m}
$$ 
are in $L^1_M$ in the case when~$\g$ is less than~$d+1$.   As $f_\gamma$ is   continuous  and bounded away from any neighborhood of $\wh 0,$ it suffices to prove  that 
\beq
\label {dimhomoconcretwhHdemoeq1}
\sum_{n\in \N^d} \int_{|\lam| (2|n|+d)\leq 1}  \bigl(|\lam| (2|n|+d)\bigr) ^{-\g}
|\lam|^d d\lam <\infty.
\eeq
Now, performing the change of variables $ \lam' = \lam (2|n|+d)$, we find out that  
$$
\sum_{n\in \N^d} \int_{|\lam| (2|n|+d)\leq 1}  \bigl(|\lam| (2|n|+d)\bigr) ^{-\g}
|\lam|^d d\lam  = \sum_{n\in \N^d} (2|n|+d)^{-d-1}  \int_{|\lam'|\leq 1}  |\lam'| ^{d-\g}
d\lam'.
$$
Because~$\g<d+1,$ this implies that the last integral is finite. As~$\ds
\sum_{n\in \N^d} (2|n|+d)^{-d-1} $ is finite, one may conclude that
$f_\gamma$ is in $L^1_M.$\qed
\medbreak
In order to give  an example of tempered distribution on the Heisenberg
group that is not a function, let us finally prove Proposition\refer {examplepartiefinie}.
We start with the obvious observation that   
$$
\biggl|
 \int_{\wh \H^d} 
\biggl(\frac {\theta(n,n,\lam)+ \theta(n,n,-\lam) -2\theta(\wh 0) }{ |\lam|^\g (2|n|+d)^\g} \biggr)\d_{n,m} \,d\wh w\biggr|
\leq \cI_1 +\cI_2
$$
with 
$$\begin{aligned}
\cI_1&\eqdefa \int_{\wh \H^d}  {\bf 1}_{\{|\lam|(2|n|+d)\geq 1 \}} \d_{n,m} \frac {\bigl|\theta(n,n,\lam)+ \theta(n,n,-\lam) -2\theta(\wh 0) \bigr|}{ |\lam|^\g (2|n|+d)^\g}\,d\wh w\\  \andf\cI_2&\eqdefa \int_{\wh \H^d}  {\bf 1}_{\{|\lam|(2|n|+d)<1 \}} \d_{n,m} \frac {\bigl|\theta(n,n,\lam)+ \theta(n,n,-\lam) -2\theta(\wh 0) \bigr|}{ |\lam|^\g (2|n|+d)^\g}\,d\wh w.\end{aligned}
$$
On the one hand, we have
$$
\cI_1 \leq  4\|\theta\|_{L^\infty(\wh \H^d)} \sum_{n\in \N^d} \int_{\R} {\bf 1}_{\{|\lam|(2|n|+d)\geq 1 \}}\frac {1}{ |\lam|^\g (2|n|+d)^\g}\, |\lam|^d d\lam.$$
Changing variable~$\lam' = \lam (2|n|+d)$ gives 
$$
\cI_1 \leq 4\|\theta\|_{L^\infty(\wh \H^d)} \sum_{n\in \N^d} \frac {2}{ (2|n|+d)^{d+1} } \int_1^\infty  |\lam'|^{d-\g} \, d\lam'.
$$
As~$\g$ is greater than~$d+1$, the integral in~$\lam'$ is finite and we get
\beq
\label {examplepartiefiniedemoeq1}
\cI_1\leq C\|\theta\|_{L^\infty(\wh \H^d)}.
\eeq
On the other hand,  changing again variable~$\lam'= \lam|(2|n|+d)$, we see that
$$
\cI_2  = \!\!\!\!\sum_{n\in \N^d}  \frac {2}{ (2|n|+d)^{d+1} } \! \int_0^1 \biggl( \theta\Bigl  (n,n,\frac {\lam'} {2|n|+d}\Bigr) +\theta  \Bigl (n,n,\frac {-\lam'} {2|n|+d}\Bigr)  -2\theta(\wh 0)\biggr) |\lam'|^{d-\g}  d\lam'.
$$

At this stage, we need a suitable  bound of the integrand just above.  This will be achieved thanks to the following lemma. 
\begin{lemma}
\label {lemmaproofexamplepartiefinie}
{\sl
There exists an integer $k$ such that for any function~$\theta$ in~$\cS(\wh \H^d)$, we have
$$
\forall (n,n,\lam)\in \wh\H^d\,,\ |\theta (n,n,\lam)-\theta (\wh 0)| \leq C  \|\theta\|_{k,k,\cS(\wh \H^d)}
 \sqrt {|\lam| (2|n|+d)}.
 $$
 }
\end{lemma}
\begin{proof}
Theorem\refer {MappingofPH} guarantees  that~$\theta$ is the  Fourier transform of a function $f$ of~$\cS(\H^d)$ (with  control  of semi-norms). Hence  it suffices to prove  that  
\begin{equation}
\label{continuityRealTopodemoeq0}
\begin{aligned}
 \Bigl|\wh f_\H(\wh w)-\d_{n,m}\int_{\H^d} f(w)\,dw\Bigr| &\leq CN(f) \Bigl(\sqrt{|\lam|(|n+m|+d)}+|\lam|\delta_{n,m}\Bigr)\with\\
N(f)  & \eqdefa \int_{\H^d} \bigl( 1+|Y| +|s+2\langle \eta,y\rangle|\bigr) |f(Y,s)| \,dw.
\end{aligned}
\end{equation}
According to \eqref {definFourierWigner}, we  have
$$\displaylines{
\wh f_\H(\wh w)-\d_{n,m}\int_{\H^d} f(w)\,dw \hfill\cr\hfill=\int_{\H^d} f(w)\biggl(e^{-i\lam(s+2\langle\eta,y\rangle)}
\int_{\R^d} e^{-2i\lam\langle\eta,z\rangle} H_{n,\lam}(z+2y)H_{m,\lam}(z)\,dz
-\int_{\R^d}H_{n,\lam}(z) H_{m,\lam}(z)\,dz\biggr)dw.}
$$
The right-hand side  may be decomposed into $I_1+I_2+I_3$ with
$$
\begin{aligned}
I_1&=\int_{\H^d} e^{-i\lam(s+2\langle\eta,y\rangle)}f(w)
\biggl(\int_{\R^d} \Bigl(e^{-2i\lam\langle\eta,z\rangle}-1\Bigr)H_{n,\lam}(z+2y)H_{m,\lam}(z)\,dz\biggr)dw,\\
I_2&=\int_{\H^d} e^{-i\lam(s+2\langle\eta,y\rangle)}f(w)
\biggl(\int_{\R^d} \bigl(H_{n,\lam}(z+2y)-H_{n,\lam}(z)\bigr)H_{m,\lam}(z)\,dz\biggr)dw\andf\\
I_3&=\int_{\H^d} \biggl(e^{-i\lam(s+2\langle\eta,y\rangle)}-1\biggr)f(w)\biggl(\int_{\R^d} H_{n,\lam}(z)H_{m,\lam}(z)\,dz\biggr)dw.
\end{aligned}
$$
To bound $I_1,$ it suffices to use that
$$
\biggl|\int_{\R^d} \Bigl(e^{-2i\lam\langle\eta,z\rangle}-1\Bigr)H_{n,\lam}(z+2y)H_{m,\lam}(z)\,dz\biggr|
\leq 2\sum_{j=1}^d |\lam\eta_j| \int_{\R^d} |H_{n,\lam}(z+2y)|\,|z_j H_{m,\lam}(z)|\,dz,
$$
whence, combining Cauchy-Schwarz inequality and \eqref {relationsHHermiteCAb}, 
$$
\biggl|\int_{\R^d} \Bigl(e^{-2i\lam\langle\eta,z\rangle}-1\Bigr)H_{n,\lam}(z+2y)H_{m,\lam}(z)\,dz\biggr|
\leq \sum_{j=1}^d |\eta_j| \sqrt{|\lam|(4m_j+2)}.
$$
This gives
\begin{equation}\label{eq:I1}
|I_1|\leq \sqrt{|\lam|(4|m|+2d)}\int_{\H^d} |\eta|\,|f(y,\eta,s)|\,dy\,d\eta\,ds.
\end{equation}
To handle the term $I_2,$ we use the following mean value formula:
$$
H_{n,\lam}(z+2y)-H_{n,\lam}(z)=2 y\cdot\int_0^1\nabla H_{n,\lam}(z+2ty)\,dt,
$$
which implies, still using  \eqref {relationsHHermiteCAb},
$$
\biggl| \int_{\R^d} \bigl(H_{n,\lam}(z+2y)-H_{n,\lam}(z)\bigr)H_{m,\lam}(z)\,dz\biggr| 
\leq \sum_{j=1}^d |y_j| \sqrt{|\lam|(4n_j+2)},$$
and thus 
\begin{equation}\label{eq:I2}
|I_2|\leq \sqrt{|\lam|(4|n|+2d)}\int_{\H^d} |y|\,|f(y,\eta,s)|\,dy\,d\eta\,ds.
\end{equation}
Finally, it is clear that the mean value theorem (for the exponential function) and the fact
that $(H_n)_{n\in\N^d}$ is an orthonormal family imply that 
\begin{equation}\label{eq:I3}
|I_3|\leq |\lam| \delta_{n,m}\int_{\H^d} |s+2\langle\eta,y\rangle|\,|f(y,\eta,s)|\,dy\,d\eta\,ds.
\end{equation}
Putting \eqref{eq:I1}, \eqref{eq:I2} and \eqref{eq:I3} together ends the proof of the lemma.
\end{proof}

\medbreak

It is now easy to complete the proof of   Proposition\refer   {examplepartiefinie}. 
Indeed, taking~$\ds \lam = \pm\frac {\lam'} {2|n|+d}$ in Lemma\refer  {lemmaproofexamplepartiefinie},
we discover that
 $$
  \biggl | \theta\Bigl  (n,n,\frac {\lam'} {2|n|+d}\Bigr) +\theta  \Bigl (n,n,-\frac {\lam'} {2|n|+d}\Bigr)  -2\theta(\wh 0)\biggr | \leq C \|\theta\|_{k,k,\cS(\wh \H^d)}\sqrt {\lam'} .
$$
  This implies that 
\beno
  \cI_2  & \leq & C  \|\theta\|_{k,k,\cS(\wh \H^d)} \sum_{n\in \N^d}   \frac {1}{ (2|n|+d)^{d+1} }  \int_0^1 |\lam'|^{d+\frac 12-\g} \,d\lam'\\
  & \leq & C  \|\theta\|_{k,k,\cS(\wh \H^d)}  \int_0^1 |\lam'|^{d+\frac 12-\g} \,d\lam'\,.
\eeno
As~$\g<d+3/2$, combining with\refeq  {examplepartiefiniedemoeq1} completes the proof of the proposition.
  \qed


\section {Examples of  computations of Fourier transforms} 
\label {computeFHcS'}

The present section aims at pointing out a few examples of computations
of Fourier transform that may be easily achieved within our approach. 

Let us start with Proposition\refer {Fourierdetaand1}.
The first identity is easy to prove. Indeed, according to\refeq{definWigner}, we have
\beno
\langle \cF_{\H}(\d_0),\theta\rangle _{\cS'(\wh \H ^d)\times\cS(\wh \H ^d)}  & = &  \langle \d_0,{}^t \cF_{\H}(\theta)\rangle _{\cS'(\H ^d)\times\cS(\H ^d)} \\
& = & \int_{\wh \H ^d} (H_{m,\lam}|H_{n,\lam} )_{L^2} \theta (\wh w)\, d\wh w. 
\eeno
As~$\bigl(H_{n,\lam}\bigr)_{n\in \N}$ is an orthonormal basis of~$L^2(\R^d)$, we get
$$
\langle \cF_{\H}(\d_0),\theta\rangle _{\cS'(\wh \H ^d)\times\cS(\wh \H ^d)}  = \sum_{n\in \N^d}\int_{\R} \theta (n,n,\lam)\,|\lam|^dd\lam
$$ 
which is exactly the first identity. 

\medbreak

For proving the second identity, we start again from the definition of the Fourier transform on~$\cS'(\H^d),$  and get 
\beq
\label {Fourierdeltaet1Hdemoeq1}
\langle \cF_ {\H}({\bf 1}),\theta\rangle _{\cS'(\wh \H ^d)\times\cS(\wh \H ^d)} =\int_{\H^d}( {}^t\cF_ {\H}\theta)(w) \,dw.
\eeq
Let us underline that because~${}^t\cF_{\H}\theta$ belongs to~$\cS(\H^d)$, 
the above integral makes sense. Besides, \refeq{tFequivF-1} implies that
$$
\langle \cF_{\H}({\bf 1}),\theta\rangle _{\cS'(\wh \H ^d)\times\cS(\wh \H ^d)} = \frac {\pi^{d+1}} {2^{d-1}}
\int_{\H^d} (\cF_{\H}^{-1} (\theta))(y,-\eta,-s) \,dy\,d\eta\, ds. 
$$
By Theorem\refer {FourierL1basicbis} and Lemma\refer {lemmaproofexamplepartiefinie} we have, for any integrable function $f$ on $\H^d,$
$$
\ds \wh f_\H(\wh 0)= \int_{\H^d} f(w)\,dw.
$$ Thus we get
$$
\langle \cF_{\H}({\bf 1}),\theta\rangle _{\cS'(\wh \H ^d)\times\cS(\wh \H ^d)} = \frac {\pi^{d+1}} {2^{d-1}}
\cF_{\H} \bigl(\cF_{\H}^{-1} (\theta)\bigr)(\wh 0)= \frac {\pi^{d+1}} {2^{d-1}}\theta (\wh 0). 
$$
This concludes the proof of the proposition.\qed

\medbreak

In order to prove Theorem\refer{Fourierhorizontal+}, we need to  establish the following continuity property of the Fourier transform.
\begin{proposition}
\label {contnuityFHcS'}
{\sl
Let~$\suite T n \N$ be a sequence of tempered distribution on~$\H^d$ which converges to~$T$ in~$\cS'(\H^d)$. Then the sequence~$\suite {\cF_\H T}  n \N$ converges to~$\cF_\H T$ in~$\cS'(\wh\H^d)$.
}
\end{proposition}
\begin{proof}
By definition of the Fourier transform on~$\H^d$, we have
$$
\forall \theta \in \cS(\wh \H^d)\,,\ \langle \cF_\H T_n , \theta\rangle_{\cS'(\wh\H^d)\times \cS(\wh\H^d)}
=\langle T_n, {}^t\cF_\H \theta\rangle _{\cS'(\H^d)\times \cS(\H^d)}
$$
Since~$\suite T n \N$ converges to~$T$ in~$\cS'(\H^d),$ we have 
$$
\forall \theta \in \cS(\wh \H^d)\,,\  \lim_{n\rightarrow \infty}
\langle T_n, {}^t\cF_\H \theta\rangle _{\cS'(\H^d)\times \cS(\H^d)} = \langle T, {}^t\cF_\H \theta\rangle _{\cS'(\H^d)\times \cS(\H^d)}.
$$ 
Therefore, putting the above two relations together eventually yields
$$
\forall \theta \in \cS(\wh \H^d)\,,\  \lim_{n\rightarrow \infty} \langle \cF_\H T_n , \theta\rangle_{\cS'(\wh\H^d)\times \cS(\wh\H^d)}
=\langle T, {}^t\cF_\H \theta\rangle _{\cS'(\H^d)\times \cS(\H^d)}
= \langle \cF_\H T , \theta\rangle_{\cS'(\wh\H^d)\times \cS(\wh\H^d)}.
$$
This concludes the proof of the proposition.
\end{proof}

\medbreak 

Now, proving  Theorem\refer {Fourierhorizontal+} just amounts to recast Theorem~1.4 of\ccite{bcdFHspace} (and its proof)  in terms of tempered distributions.  We recall  it here for the reader convenience.
\begin{theorem}
\label {Fourierhorizontal'}
{\sl
Let $\chi$  be a function of~$\cS(\R)$ with value~$1$ at~$0,$ and  compactly supported Fourier transform.
 Then for any  function~$g$ in~$L^1(T^\star \R^d)$ and  any sequence~$\suite \e n \N$ tending to~$0$, we have
\begin{equation}\label{eq:horizontal}
\lim_{n\rightarrow \infty} \cF_\H(g\otimes \chi(\e_n\cdot)) = 2\pi (\cG_\H g)  \mu_{\wh\H_0^d}
\end{equation}
in the sense of measures on~$\wh\H^d$. 
}
\end{theorem}

Because $g\otimes \chi(\e_n\cdot)$ tends to $g\otimes{\bf 1}$ in $\cS'(\H^d),$  Proposition\refer {contnuityFHcS'} guarantees  that
\beq
\label {Fourierhorizontal+demoeq1}
 \cF_\H(g\otimes{\bf 1})  = \lim_{n\rightarrow \infty} \cF_\H(g\otimes\chi(\e_n\cdot)). 
\eeq
Moreover, according to Theorem~1.4 of\ccite{bcdFHspace}, we have, for any~$\theta$ in~$\cS(\wh\H^d)$, 
$$
\cI_{\e_n} (g, \theta)  = \int_{\wh \H^d}   \frac 1 {\e_n} \wh \chi\Bigl(\frac  \lam {\e_n}\Bigr)  G(\wh w) \theta(\wh w) \,d\wh w\quad\with 
G(\wh w)  \eqdefa  \int_{T^\star \R^d} \ov\cW(\wh w,Y)  g(Y) \,dY.
$$
As~$g$ is integrable on~$T^\star\R^d$,  Proposition~2.1 of\ccite{bcdFHspace}  implies that the (numerical) product~$G\theta$ is a continuous function that satisfies
$$
|G (\wh w)\theta(\wh w)| \leq C  \bigl( 1+|\lam|( |n+m|+d) +|n-m|\bigr) ^{-2d+1}.
$$
This matches  the hypothesis of Lemma~3.1 in\ccite {bcdFHspace}, and thus
$$
\lim_{n\rightarrow\infty} \int_{\wh \H^d}   \frac 1 {\e_n} \wh \chi\Bigl(\frac  \lam {\e_n}\Bigr)  G(\wh w) \theta(\wh w) \,d\wh w
= \int_{\wh\H^d_0} \theta (\dot x,k) (\cG_\H g)(\dot x,k)  d\mu_{\wh \H^d_0} (\dot x,k).
$$
Together with\refeq {Fourierhorizontal+demoeq1}, this proves the theorem.
\qed


\appendix 
\section{Useful tools and more results}
\label {FourierHbasic}

For the reader convenience, we here recall (and sometimes prove) some results
that have been used repeatedly in the paper. We also provide one more result concerning the
action of the Fourier transform on derivatives.

\subsection{Hermite functions}

In addition to the creation operator $C_j \eqdefa -\partial_j+M_j$ already defined in the introduction, 
 we used the following  \emph{annihilation operator}: 
\beq
\label {definCreaAnnhil}
A_j \eqdefa \partial_j+M_j.
\eeq
It is very classical (see e.g. \cite{O}) that 
\beq
\label {relationsHHermiteCA}
A_j H_{n} = \sqrt {2n_j}\,  H_{n-\d_j} \andf C_j H_n = \sqrt {2n_j+2}\, H_{n+\d_j}.
\eeq
As, obviously,
\begin{equation}\label{Mjdj}
2M_j =C_j+A_j \andf 2\partial_j =A_j-C_j,
\end{equation}
 we discover that
\beq
\label {relationsHHermiteCAb}
\begin{aligned}
M_j H_{n}  &=  \ds  \frac 12 \bigl(  \sqrt {2n_j} \,H_{n-\d_j}+\sqrt {2n_j+2}\,  H_{n+\d_j}\bigr) \andf \\
\partial_j H_n &=  \ds \frac 12\bigl (\sqrt {2n_j}\, H_{n-\d_j}- \sqrt {2n_j+2}\, H_{n+\d_j}\bigr).
\end{aligned}
\eeq

\subsection {The inversion theorem}

We here present the proof of Theorem\refer{inverseFourier-Plancherel}.
In order to establish the inversion formula, consider a function~$f$  in~$\cS(\H^d).$ 
Then we observe  that if we make the  change of variable~$x'= x-2y$
in the integral defining~$(\cF^{\H}(f)(\lam)(u))(x)$ (for any~$u$ in~$ L^2(\R^d)$) 
and use the definition of the Fourier transform with respect to the variable~$s$ in~$\R,$ then we get
\begin{eqnarray}
\label {formulaFourerHbis-1}
\bigl(\cF^{\H}(f)(\lam)(u)\bigr)(x)  \!\!\!\!\!\!\!\!&&= 
\int_{\H^d} f(y,\eta,s) e^{-i\lam s -2i\lam \langle \eta, x-y\rangle} u(x-2y)\,dy\,d\eta\, ds\nonumber\\
&&=
2^{-d}\int_{T^\star \R^d}   (\cF_s f)\Bigl(\frac {x-x'} 2,\eta, \lam\Bigr) 
e^{-i\lam \langle \eta,  x+x'\rangle} u(x') \,dx'd\eta.
\end{eqnarray}
This can be written
\begin{equation}\label {formulaFourerHbis}
\cF^{\H}(f)(\lam) (u)(x)  =  \int_{\R^d}  K_f(x,x',\lam) u(x')\, dx',
\end{equation}
$$
\begin{aligned}
\with K_f(x,x',\lam) &\eqdefa 2^{-d} \int_{(\R^d)^\star }  (\cF_s f)\Bigl(\frac {x-x'} 2,\eta, \lam\Bigr) 
e^{-i\lam \langle \eta,  x+x'\rangle} d\eta\\
&= 2^{-d} (\cF_{\eta,s} f)\Bigl(\frac {x-x'} 2,\lam(x+x'), \lam\Bigr) .
\end{aligned}
$$
This identity enables us to decompose~$\cF_\H$ into the product of three very simple operations, namely 
\begin{equation}
\label {decomF_H}
\begin{aligned} \cF_\H & =  2^{-d}P_H \circ \Phi\circ \cF_{\eta,s}\with\\
\Phi(\f)(x,x',\lam) & \eqdefa \f \Bigl(\frac {x-x'} 2,\lam(x+x'), \lam\Bigr)\andf\\
  (P_H\psi)(n,m,\lam) & \eqdefa    \bigl(\psi(\cdot,\lam)| H_{n,\lam} \otimes H_{m,\lam}\bigr)_{L^2(\R^{2d})}.
 \end{aligned}
\end{equation}
Let us point out that for all~$\lam$ in $\R\setminus\{0\},$  the map
$$
\f(\cdot ,\lam) \longmapsto \Phi(\f)(\cdot, \lam)
$$
is an automorphism of~$L^2(\R^{2d})$ such that
\beq
\label {inverseFourierL2demoeq1}
\|\Phi(\f) (\cdot,\lam)\|_{L^2(\R^{2d})} =|\lam|^{- \frac d  2} \|\f(\cdot,\lam)\|_{L^2(\R^{2d})},
\eeq
and that  the inverse of $\Phi$ is explicitly given by
\beq
\label {inverseFourierL2demoeq2}
\Phi^{-1} (y,z,\lam)= \psi \Bigl (y+\frac z {2\lam} , -y+\frac z {2\lam}, \lam \Bigr)\cdotp
\eeq

Next, Operator~$P_H$ just  associates to any  vector of $L^2(\R^{2d})$ 
its coordinates with respect to the orthonormal basis~$\bigl(H_{n,\lam}\otimes H_{m,\lam} \bigr)_{(n,m)\in \N^{2d}}.$ It is by definition  an isometric isomorphism from $L^2(\R^{2d})$ to $\ell^2(\N^{2d}),$
 with   inverse
\beq
\label {inverseFourierL2demoeq3}
(P_H^{-1} \theta)(x,x',\lam) = \sum_{(n,m)\in \N^{2d}} \theta(n,m,\lam) H_{n,\lam} (x) H_{m,\lam}(x').
\eeq

Obviously, arguing by density, Formula\refeq{decomF_H} may be extended to $L^2(\H^d).$
Therefore, according to Identities \refeq {inverseFourierL2demoeq1}--(\ref{inverseFourierL2demoeq3}), 
and thanks to  the classical Fourier-Plancherel theorem 
in~$\R^{d+1},$   the  Fourier transform~$\cF_\H$ may be seen as the composition of three invertible and
bounded operators on~$L^2,$ and we have
$$
\cF_\H^{-1} = 2^d \cF_{\eta,s}^{-1} \circ \Phi^{-1} \circ P_H^{-1}.
$$
This gives\refeq {MappingofPHdemoeq1} and\refeq {inverseFouriereq2}. For the proof of\refeq  {newFourierconvoleq1}, we refer for instance to\ccite{bcdFHspace}.  This concludes the proof of Theorem\refer {inverseFourier-Plancherel}.
\qed


\subsection{Properties related to the sub-ellipticity of~$\D_\H$} \label{subellipticity}

 Let $k$ be a nonnegative integer. Then   setting  
$$
\|u\|_{\dot H^k(\H^d)}^2 \eqdefa \sum_{\al\in \{1,\cdots,2d\}^k} \|\cX^\al u\|_{L^2}^2,
$$
we have the following well-known result (see the proof in e.g.\ccite{CCX,H}):
\begin{theorem}
\label {maximalestimatepaire}
{\sl For any positive integer~$\ell$, we have for some constant $C_\ell>0,$
$$
\|\Delta_{\H} ^\ell u\|_{L^2(\H^d)}\leq \|u\|_{\dot H^{2\ell}(\H^d)} 
\leq C_\ell\|\Delta_{\H} ^\ell u\|_{L^2(\H^d)}\, .
$$
}
\end{theorem}
This will enable us to establish the following proposition which states that  the usual semi-norms on the Schwartz class and the semi-norms using the structure of~$\H^d$ are equivalent.
 \begin{proposition}
  \label {p:schwartz}
{\sl 
 Let us introduce the notation
 $$
( M_\H f) (X,s) \eqdefa (|X|^2-  is) f(X,s)\,.
 $$
 Next, for all~$\al =(\al_0, \al_1, \cdots, \al_{2d})$  in~$\N^{1+2d}$, we define
 $$
 w^\al \eqdefa s^{\al_0} y_1^{\al_1} \cdots  y_d^{\al_d} \eta_1^{\al_{d+1}} \cdots  \eta_d^{\al_{2d}}
 \andf \wt |\al|\eqdefa 2\al_0+\al_1\cdots +\al_{2d}\,. 
 $$
Then the two families of semi-norms defined on~$\cS(\H^d)$ by
$$\begin{aligned}
 \|f\|^2_{p,\cS(\H^d )} & \eqdefa  \|f\|_{L^2} ^2+ \|M_\H^pf\|_{L^2}^2 +\| \D_\H^p f \bigr\|^2_{L^2}  \andf\\
 N_p^2(f) &  \eqdefa   \sum_{\wt |\al|+|\b| \leq p }  \bigl\| w^\al \cX^\b f \bigr\|^2_{L^2}
  \end{aligned}$$
are   equivalent to the classical family of semi-norms on~$\cS(\R^{2d+1}).$
}
 \end{proposition}

 \begin{proof}
  As obviously~$\|f\|_{p,\cS(\H^d)}\leq N_{2p} (f)$, showing  that the two 
families of   semi-norms are equivalent reduces to proving that 
 \beq
  \label {c:schwartzdemooeq3}
 \forall p\in \N\,,\  \exists (C_p,M_p)\,/\ \forall f\in \cS(\H^d)\,,\ N_p(f)\leq C_p\|f\|_{M_p,\cS(\H^d )} \, .
  \eeq
  Now, integrating by parts yields
  \beno
  \int_{\H^d} w^\al \,\cX^\b f(w) \,w^\al \cX^\b\overline f(w)\,dw =  (-1)^{|\b|} \int_{\H^d} f(w)  \cX^\b\bigl(w^{2\al}  \cX^\b \overline f(w)\bigr)dw.
  \eeno
  Observe that~$\cX^\g w^{\g'}$ is either null or an homogeneous polynomial (with respect to the dilations \eqref {defdilations}) of degree~$\g'-\g$, and   equal to~$0$
 if the length of~$\g$ is greater than the length of~$\g'$. Thus, thanks to Leibniz' rule, we have
 \beq
  \label {c:schwartzdemooeq0}
[\cX^\beta,w^{2\al}] f(w) = \sumetage {  \wt |\al'|\leq 2\wt |\al|-1} {|\b'|\leq|\b|-1}  a_{\al,\al',\b',\b} \,w^{\al'} \cX^{\b'}f(w).
  \eeq
  Hence we get that 
$$
 \int_{\H^d} f(w)  \cX^\b\bigl(w^{2\al}  \cX^\b \overline f(w)\bigr)dw =
 \sumetage {\wt|\al'|\leq 2 \wt |\al|} {|\b'|\leq |\b|} a_{\al,\al',\b,\b'} \int_{\H^d} w^{\al'}  f(w)  \cX^{\b'} \cX^\b\overline f(w) \,dw.
 $$
 Thanks to Cauchy-Schwarz inequality and by definition of~$M_\H$, we get, applying Theorem\refer {maximalestimatepaire} and taking $p$ large enough,
 $$
 \sumetage {\wt|\al'|\leq 2 \wt |\al|} {|\b'|\leq |\b|} a_{\al,\al',\b,\b'} \int_{\H^d} w^{\al'}  f(w)  \cX^{\b'} \cX^\b\overline f(w)\, dw 
 \leq  C\bigl(\|f\|_{L^2}^2 + \|M_\H^{p} f\|_{L^2} ^2+ \|\D_\H^{p} f\|_{L^2}^2\bigr)\cdotp
 $$
 This proves that the two families of semi-norms in the above statement are equivalent.
\medbreak
In order to establish that they  are also equivalent to the classical family, one 
can observe that for all $j$ in~$ \{1, \cdots,d \},$ 
$$
S = \frac 1 4 [\Xi_j,\cX_j],\quad \partial_{y_j}=  \cX_j -  \frac {\eta_j}  2  \,(\Xi_j\cX_j-  \cX_j \Xi_j) \andf \partial_{\eta_j}= \Xi_j+ \frac {y_j}  2  \,(\Xi_j\cX_j-  \cX_j \Xi_j),
$$ 
from which we  easily infer that
$$  
\wt \|f\|_{p,\cS(\R^{2d+1})}
 \leq C N_{2p}(f)
 \quad\hbox{with }\ \wt \|f\|_{p,\cS(\R^{2d+1})}^2 \eqdefa \sum_{|\al|+|\b|\leq p} \|x^\al \partial^\b f \|_{L^2(\R^{2d+1})}^2.$$  This  ends the proof of the proposition.
\end{proof}


\subsection {Derivations and multiplication in the space~$\cS(\wh \H ^d)$}\label {derivationsmultiplication}
In Section \ref {main}, we only considered  the effect of the Laplacian~$\D_\H$ or of the derivation~$\partial_s$  on Fourier transform. Those operations led to multiplication by~$-4|\lam| (2|m|+d)$ or~$i\lam,$ respectively,   of the Fourier transform.  We also studied the effect of the multiplication by~$|Y|^2$ 	or~$-is,$ 
and found out that they  
correspond  to the `derivation operators'~$\wh \D$ and~$\wh \cD_\lam$
for functions on $\wt\H^d.$   

Our purpose here is to study the effect of left invariant differentiations~$\cX_j$ and $\Xi_j$ and multiplication by~$M_j^{\pm}\eqdefa y_j\pm i \eta_j$ on the Fourier transform. 
This is  described by the following proposition.
 \begin{proposition}
 \label {actionX_jonFH}
{\sl  For any function~$f$ in~$\cS(\H^{d})$,   we have, for~$\lam$ different from~$0$,
\beno
\cF_\H \cX_j f & =&   -\wh \cM^+_j  \wh f_\H \andf  (\cF_\H \Xi_j f)   =   -\wh \cM^-_j  \wh f_\H\with\\
 \wh \cM^+_j \theta (\wh w)  & \eqdefa &  |\lam|^{\frac 12 } \bigl( \sqrt {2m_j+2} \,\theta(n,m+\d_j,\lam) -\sqrt {2m_j} \,\theta(n,m-\d_j,\lam)\bigr)\andf\\
 \wh \cM^-_j \theta (\wh w)   &\eqdefa & \frac { i \lam }{ |\lam|^{\frac 12 }}  \bigl( \sqrt {2m_j+2}\, \theta(n,m+\d_j,\lam) +\sqrt {2m_j} \,\theta(n,m-\d_j,\lam)\bigr).
\eeno
We also have~$\cF_\H  M_j^\pm f  =  \wh \cD_j^\pm \wh f_\H$ with
\beno
(\wh \cD_j^\pm \theta)(\wh w) &  \eqdefa & \frac {1_{\{\pm\lam>0\}} }{2|\lam|^{\frac 12} } \Bigl( \sqrt {2n_j}\, \theta (n-\d_j,m,\lam)-\sqrt  {2m_j\!+\!2}\, \theta(n,m+\d_j,\lam)\Bigr)\\
&&\  {}+  \frac {1_{\{\pm\lam<0\}} } {2|\lam|^{\frac 12}} \Bigl( \sqrt {2n_j\!+\!2}\, \theta(n+\d_j,m,\lam) -\sqrt  {2m_j}\, \theta(n,m-\d_j,\lam\Bigr).
\eeno}
 \end{proposition}

 \begin{proof}
The main point is to  compute
$$
\partial_{y_j}  \cW(\wh w,Y)\,,\partial_{\eta_j} \cW(\wt w,Y)\,,\ y_j \cW(\wh w,Y)\andf \eta_j \cW(\wh w,Y).
$$
By the definition of~$\cW$ and Leibniz formula, we have, using 
the notation $f_\lam(x)\eqdefa f(|\lam|^{1/2}x),$ 
$$
\partial_{y_j} \cW(\wh w,Y) \!= \int_{\R^d} e^{2i\lam\langle z,\eta\rangle } |\lam|^{\frac 12} 
\bigl( (\partial_jH_n)_{\lam} (y+z) H_{m,\lam} (-y+z)  -H_{n,\lam} (y+z) (\partial_jH_m)_\lam (-y+z) \bigr)dz.
$$
{}From\refeq {relationsHHermiteCAb}, we infer that 
\begin{multline}\label{actionX_jonFHdemoeq1}
\partial_{y_j} \cW(\wh w,Y) \!=  \frac {|\lam|^{\frac 12}} 2  \bigl( \sqrt {2n_j}\, \cW(n-\d_j,m,\lam,Y) \bigr) 
-\sqrt {2n_j+2}\, \cW(n+\d_j,m,\lam,Y) \\
-\sqrt {2m_j}\,\cW(n,m-\d_j,\lam,Y) +\sqrt {2m_j+2}\,\cW(n,m+\d_j,\lam,Y)\bigr).
 \end{multline}
Let us observe that 
\beno
\partial_{\eta_j} \cW(\wh w,Y)  & = & \int_{\R^d} 2i\lam z_j e^{2i\lam\langle \eta,z\rangle} H_{n,\lam}(y+z)H_{m,\lam} (-y+z) dz\\
& = & i\lam \int_{\R^d} e^{2i\lam\langle \eta,z\rangle} \bigl ( y_j+z_j) H_{n,\lam}(y+z)H_{m,\lam} (-y+z) \\
&& \qquad\qquad\qquad\qquad\qquad\quad{}
+ H_{n,\lam}(y+z)(-y_j+z_j)H_{m,\lam} (-y+z)\bigr)\,dz\\
&= & \frac {i\lam} {|\lam|^{\frac 12} }\int_{\R^d}e^{2i\lam\langle \eta,z\rangle} \bigl((M_jH_{n})_\lam(y+z)H_{m,\lam} (-y+z) \\
&& \qquad\qquad\qquad\qquad\qquad\qquad\qquad{}+H_{n,\lam}(y+z) (M_jH_m)_\lam(-y+z) \bigr)\, dz.
\eeno
Now, using again\refeq {relationsHHermiteCAb}, we get 
\begin{multline}\label{actionX_jonFHdemoeq2}
\partial_{\eta_j} \cW(\wh w,Y) \!=   \frac {i\lam} {2|\lam|^{\frac 12} }  \bigl( \sqrt {2n_j}\, \cW(n-\d_j,m,\lam,Y) \bigr) 
+\sqrt {2n_j+2}\, \cW(n+\d_j,m,\lam,Y) \\
+\sqrt {2m_j}\,\cW(n,m-\d_j,\lam,Y) +\sqrt {2m_j+2}\,\cW(n,m+\d_j,\lam,Y)\bigr).
 \end{multline}
For multiplication by~$y_j$, we proceed along the same lines. By definition of~$\cW$, we have
$$
y_j\cW(\wh w,Y) \!=\! \frac 1 {2|\lam|^{\frac12}}\! \int_{\R^d}\! e^{2i\lam\langle \eta,z\rangle}\!   
\bigl( (M_jH_n)_\lam (y+z) H_{m,\lam}(-y+z) - H_{n,\lam}(y+z) (M_jH_m)_\lam(-y+z)\bigr)dz.
$$
Still using\refeq {relationsHHermiteCAb}, we deduce that 
\begin{multline}\label{actionX_jonFHdemoeq3}
y_j \cW(\wh w,Y)  =    \frac 1 {4|\lam|^{\frac 12} }  \bigl( \sqrt {2n_j} \,\cW(n-\d_j,m,\lam,Y) \bigr) 
+\sqrt {2n_j+2}\,\cW(n+\d_j,m,\lam,Y) \\-\sqrt {2m_j}\,\cW(n,m-\d_j,\lam,Y) -\sqrt {2m_j+2}\,\cW(n,m+\d_j,\lam,Y)\bigr).
 \end{multline}
For the multiplication by~$\eta_j$, let us observe that, performing an integration by parts, we can~write
\beno
\eta_j\cW(\wh w,Y)  & = & \frac 1 {2i\lam} \int_{\R^d} \partial_{z_j} \bigl(e^{2i\lam\langle \eta,z\rangle} \bigr)
H_{n,\lam} (y+z) H_{m,\lam}(-y+z)\, dz\\
& = & \frac i {2\lam} \int_{\R^d} e^{2i\lam\langle \eta,z\rangle}\partial_{z_j}  \bigl(
H_{n,\lam} (y+z) H_{m,\lam}(-y+z) \bigr)\,dz.
\eeno
Leibniz formula implies that 
$$
\longformule
{
\eta_j\cW(\wh w,Y) =  \frac {i|\lam|^{\frac 12}}  {2\lam} \int_{\R^d} e^{2i\lam\langle \eta,z\rangle} \bigl(
(\partial_jH_{n})_\lam (y+z) H_{m,\lam}(-y+z) 
}
{ {}
+  H_{n,\lam} (y+z) (\partial_j H_m)_\lam (-y+z) \bigr)dz.
}
$$
Using\refeq {relationsHHermiteCAb}, we deduce that 
\beq
 \label {actionX_jonFHdemoeq4}
\begin{split}
&\eta_j \cW(\wh w,Y)  =    \frac {i |\lam|^{\frac 12} } {4\lam}  \bigl( \sqrt {2n_j} \,\cW(n-\d_j,m,\lam,Y) \bigr) 
-\sqrt {2n_j+2}\, \cW(n+\d_j,m,\lam,Y) \\
 &\qquad\qquad\qquad\qquad{}+\sqrt {2m_j}\,\cW(n,m-\d_j,\lam,Y) -\sqrt {2m_j+2}\,\cW(n,m+\d_j,\lam,Y)\bigr).
 \end{split}
\eeq
As we have~$e^{-is\lam} \cX_j\bigl( e^{is\lam}\cW(\wh w,Y) \bigr) = 2i\eta_j \lam\cW(\wh w,Y) +\partial_{y_j} \cW(\wh w,Y),$
we infer from\refeq {actionX_jonFHdemoeq1} and\refeq{actionX_jonFHdemoeq4} that
\beq
 \label {actionX_jonFHdemoeq5}
 \begin{split}
 e^{-is\lam} \cX_j\bigl(  e^{is\lam}\cW(\wh w,Y) \bigr)  & = |\lam|^{\frac 12} \bigl(-\sqrt {2m_j}\,\cW(n,m-\d_j,\lam,Y)\\
 &\qquad\qquad\qquad\qquad{}+
 \sqrt {2m_j+2}\,\cW(n,m+\d_j,\lam,Y)\bigr)\\
 & =  \cM_j^+\cW(\wh w,Y).
 \end{split}
 \eeq
 As we have
$$
e^{-is\lam} \Xi_j\bigl(  e^{is\lam}\cW(\wh w,Y) \bigr) = -2iy_j \lam\cW(\wh w,Y) +\partial_{\eta_j} \cW(\wh w,Y),
$$
we infer from\refeq {actionX_jonFHdemoeq2} and\refeq{actionX_jonFHdemoeq3} that
\beq
 \label {actionX_jonFHdemoeq6}
 \begin{split}
 e^{-is\lam} \Xi_j\bigl(  e^{is\lam}\cW(\wh w,Y) \bigr)  & = \frac {i\lam} {|\lam|^{\frac 12}} \bigl(\sqrt {2m_j} \cW(n,m-\d_j,\lam,Y)\\
 &\qquad\qquad\qquad\qquad{}+
 \sqrt {2m_j+2} \cW(n,m+\d_j,\lam,Y)\bigr)\\
 & = \cM_j^-\cW(\wh w,Y).
 \end{split}
 \eeq
 It is obvious that\refeq{actionX_jonFHdemoeq3} and\refeq{actionX_jonFHdemoeq4} give
 $$\displaylines{\quad
 (y_j\pm i\eta_j) \cW(\wh w,Y)   = \frac1{4|\lam|^{\frac 12}}\Bigl(\sqrt {2n_j}\,\cW(n-\d_j,m,\lam,Y) \bigr) 
+\sqrt {2n_j+2}\,\cW(n+\d_j,m,\lam,Y)\hfill\cr\hfill
-\sqrt {2m_j}\, \cW(n,m-\d_j,\lam,Y)
- \sqrt {2m_j+2}\,\cW(n,m+\d_j,\lam,Y)
 \hfill\cr\hfill\pm{\rm sgn} (\lam) 
 \bigl( \sqrt {2n_j}\,\cW(n-\d_j,m,\lam,Y)-\sqrt {2n_j+2}\,\cW(n+\d_j,m,\lam,Y)\hfill\cr\hfill
 +\sqrt {2m_j}\,\cW(n,m-\d_j,\lam,Y)-\sqrt {2m_j+2}\,\cW(n,m+\d_j,\lam,Y)\bigr) \Bigr)\cdotp\quad}
 $$
 By definition of $\cF_\H,$  this gives the result.
\end{proof}

\end{document}